\documentclass[final]{siamltex}

\usepackage{amsmath}
\usepackage{amssymb} 
\usepackage{amsfonts} 
\usepackage{graphicx}
\usepackage{overpic}
\usepackage{subfigure}
\usepackage{multirow}
\usepackage{placeins}            
\usepackage{color}
\usepackage{array}               
\usepackage{booktabs}            
\usepackage{tabularx}

\def\Re{\text{Re}}
\def\Im{\text{Im}}

\def\real{{\mathbb R}}
\def\complex{{\mathbb C}}
\def\natural{{\mathbb N}}

\def\F{\mathcal{F}}
\def\FL{F^{}_L}
\def\xL{x^{}_L}

\def\order{\mathcal{O}}
\def\ordera{\mathcal{O}_{\hspace{-1pt} a}^{}}
\def\psii{\psi_{}^{-1}}
\def\psiE{\psi_\text{\tiny E}^{}}

\def\phiE{\varphi_\text{\tiny E}^{}}

\def\psiEi{\psi_\text{\tiny E}^{-1}}

\def\phiEi{\varphi_\text{\tiny E}^{-1}}

\def\phiS{\varphi_\text{\tiny S}^{}}
\def\psiS{\psi_\text{\tiny S}^{}}
\def\phiSi{\varphi_\text{\tiny S}^{-1}}
\def\psiSi{\psi_\text{\tiny S}^{-1}}
\def\Pn{P^{(n)}_{}}
\def\pnL{p_L^{(n)}}
\def\P{\mathcal{P}}

\def\R{\mathcal{R}}
\def\S{\mathcal{S}}

\def\s{\hspace{-2pt}}

\def\ee{\hspace{-2pt}=\hspace{-2pt}}

\usepackage[normalem]{ulem}

\DefineNamedColor{named}{JungleGreen} {cmyk}{0.99,0,0.52,0}

\DefineNamedColor{named}{Purple}{cmyk}{0.45,0.86,0,0}

\title{New exponential variable transform methods for functions with endpoint singularities}

\author{Ben Adcock%
        \thanks{Department of Mathematics,	
		Purdue University,
		150 N.~University Street,
		West Lafayette, 
		IN 47907,
		USA
                  (\texttt{adcock@purdue.edu},
                   \texttt{http://http://www.math.purdue.edu/~adcock/}).
                }
    \and
          Mark Richardson%
        \thanks{Oxford University Mathematical Institute,
                24--29 St Giles', 
                Oxford OX1 3LB, UK
                (\texttt{mark.richardson@maths.ox.ac.uk},
                \texttt{http://people.maths.ox.ac.uk/richardsonm/}).
                }}

\begin{document}

\maketitle

\begin{abstract} 
The focus of this article is the approximation of functions which are analytic on a compact interval except at the endpoints.  Typical numerical methods for approximating such functions depend upon the use of particular conformal maps from the original interval to either a semi-infinite or an infinite interval, followed by an appropriate approximation procedure on the new region.  We first analyse the convergence of these existing methods and show that, in a precisely defined sense, they are sub-optimal.  Specifically, they exhibit poor \emph{resolution} properties, by which we mean that many more degrees of freedom are required to resolve oscillatory functions than standard approximation schemes for analytic functions such as Chebyshev interpolation. 

To remedy this situation, we introduce two new transforms; one for each of the above settings.  We provide full convergence results for these new approximations and then demonstrate that, for a particular choice of parameters, these methods lead to substantially better resolution properties.  Finally, we show that optimal resolution power can be achieved by an appropriate choice of parameters, provided one forfeits classical convergence.  Instead, the resulting method attains a finite, but user-controlled accuracy specified by the parameter choice.
\end{abstract}



\pagestyle{myheadings}
\thispagestyle{plain}
\markboth{B. ADCOCK AND M. RICHARDSON}
{NEW EXPONENTIAL VARIABLE TRANSFORM METHODS}

\section{Introduction}\label{sect:intro}  

The modern practitioner in scientific computing has an abundance of numerical methods at their disposal for approximating and computing with analytic functions on an interval. Most notable amongst these are Fourier series for periodic functions and Chebyshev series for the more general situation~\cite{BoydBook,TrefethenATAP}. In each of these cases, the convergence for analytic funtions is known to be geometric; that is, the error decreases like $\order(C^{-n})$ as $n \to \infty$ where $n$ is the number of degrees of freedom in the approximation, and $C > 1$.

The focus of this article is a particular aspect of the modified situation whereby a function is analytic on an interval, except possibly at the endpoints.  The techniques we shall describe in this paper fall into the broad class of what might be called \emph{variable transform methods} (see the work of Stenger and others~\cite{LundBowers,StengerSiam81,
StengerBook93,StengerBook10}). The basic idea is as follows: given a function $f(x)$ defined on $[0,1]$, and an invertible mapping $\psi:(0,1) \mapsto (-\infty,\infty)$, let $\F(s) = f(\psii(s))$, a function defined on the real line. Under relatively mild continuity assumptions on $f$, common choices of the transformation $\psi$ result in transplanted functions $\F$ which decay at least exponentially fast to their respective limiting values as $s \to \pm \infty$. Thus, by selecting a subinterval of the real $s$-axis, say $[-L,L]$ for some $L > 0$, and approximating $\F$ there, one can expect to obtain a good approximation to $\F$ on the whole real line.

The purpose of this paper is twofold.  First, we describe and analyse a particular phenomenon observed in existing variable transform methods~\cite{RichardsonTrefethen11}. This is that, in a sense which we will soon describe precisely, the standard transformations used to transplant $f(x)$ to $\F(s)$ lead to poor \textit{resolution} properties.  Loosely speaking, this means that they require many extra degrees of freedom to resolve oscillatory functions with endpoint singularities than is required in the case of analytic oscillatory functions.

With a view to improving this situation, the second purpose of this paper is to introduce several new maps.  Each new map depends on a parameter $\alpha$, which, as we show can be chosen in such a way so as to deliver far better resolution properties.  We also provide a rigorous convergence analysis of these mappings, from which we conclude that these new mappings offer both similar convergence to the existing maps and vastly improved resolution power.  Moreover, in the final part of the paper, we show that it is possible to obtain formally optimal resolution properties by choosing $\alpha$ in such a way that classical convergence is forfeited for convergence down to a finite, but user-controlled, maximal accuracy.

\subsection{Variable transform methods}
Given a function defined on an arbitrary bounded interval, it is always possible to obtain an affine transform which transplants the function to any other finite interval. Therefore, in this work, without loss of generality, we shall consider functions $f(x)$ continuous on $[0,1]$ and analytic on at least $(0,1)$. The reason for using $[0,1]$, rather than, say, the more usual $[-1,1]$, has to do with the relative densities of representable numbers in floating-point arithmetic. In short, singularities are much better handled if they happen to be at $x=0$, rather than away from zero. This setup is standard in the literature; see for example~\cite{RichardsonTrefethen11,StengerBook10}.

In general, variable transform methods constitute two key steps: i) transplantation of the function $f$ to the infinite interval $(-\infty,\infty)$ using a map $\psi$; and, ii) approximation of the transplant $\F$ using a suitable infinite interval basis set. Sinc functions are typically a sensible choice for the latter. We note, however, that in certain circumstances it is possible to utilise transformations which do not map to the infinite interval, but to rather a semi-infinite interval, such as $(-\infty,0]$. This may be a desirable strategy to use if, for example, the function $f$ is analytic on $(0,1]$, rather than just on $(0,1)$. We will investigate both settings in this paper.

Regarding the choice of basis, we note that in the semi-infinite interval setting, one cannot simply use $\mathrm{sinc}$ functions to approximate $\F$, since the transplant is not defined on $\real$. As in the related works of Boyd~\cite{Boyd82} and Richardson~\cite{Richardson13}, the approach we will take is to use Chebyshev interpolants. These are certainly not the only option. However their familiarity to practioners and ubiquity, for example, in packages such as Chebfun~\cite{Chebfun}, make them a sound choice.  In the specific instance of this paper, they also allow us to make a thorough theoretical comparison between different mappings, using the classical theory of Bernstein for polynomial approximation.

We will also elect to use Chebyshev interpolants in the infinite interval setting. This is slightly unnatural, since $\mathrm{sinc}$ functions are the optimal basis on the real line and outperform Chebyshev interpolants by an asymptotic factor of $\pi/2$ in the number of degrees of freedom required to represent a function to a fixed precision. Boyd has also shown that Fourier domain truncation is superior to Chebyshev domain truncation \cite{Boyd88}.  Indeed, we recommend the use of either $\mathrm{sinc}$ functions or Fourier interpolants if it is one's goal to achieve the most efficient approximation possible.  As stated above, our reason for using Chebyshev interpolants uniformly for all our analysis is in order to help provide expository clarity between the different sections of the paper. We note, however, that at least for the purposes of the present investigation, the choice of basis is somewhat moot. The focus of this paper is the study of the relative performance of the different maps once the underlying approximation has been fixed.

\subsection{Resolution power}\label{sect:res} 
The resolution power of a numerical scheme has traditionally been assessed by studying the complex exponential $e^{2\pi i \omega x}$. This approach was first advocated by Gottlieb and Orszag~\cite{GottliebOrszag}, with subsequent investigations including~\cite{AdcockHuybrechs,WeidemanTrefethen88}. This strategy has the benefit of providing a very clear quantitative measure of a numerical scheme -- the number of \emph{points per wavelength} (ppw) required to resolve an oscillatory function -- and therefore provides a direct way of comparing different methods.

\begin{definition} Let $\{ \Psi^{(n)}_{} \}_{n \in \natural}$ denote an approximation scheme under which $\Psi^{(n)}_{}(f)$ involves $n$ pointwise evaluations of $f$ for any function $f$. Then given  $\delta \in (0,1)$ and $\omega \in \real$, the $\delta$-resolution of $\{ \Psi^{(n)}_{} \}_{n \in \mathbb{N}} $ is the function
\begin{align}\label{eqn:epsilonResolution}
 \R(\omega;\delta) = \min \left\{ n \in \natural : \| e^{2 \pi i \omega x} - \Psi^{(n)}_{}(e^{2 \pi i \omega x})  \|_{x \in [0,1]} < \delta \right\}.
\end{align}
\end{definition}
(Here, and elsewhere in the paper, $\|\cdot\|$ denotes the $\infty$-norm.)
$\R(\omega;\delta)$ is therefore the minimum number of function samples required by the approximation scheme $\{\Psi^{(n)}_{}\}_{n \in \mathbb{N}}$ to resolve the complex exponential $e^{2 \pi i \omega x}$ to within an error of $\delta$ on $[0,1]$. 

\begin{definition} The resolution constant $r \in [0,\infty]$ of $\{ \Psi^{(n)}_{} \}_{n \in \mathbb{N}} $ is 
\begin{align}\label{eqn:resolutionConstant}
r  = \limsup_{\delta \to 1^{-}} \limsup_{|\omega| \to \infty} \frac{\R(\omega;\delta)}{|\omega|}.
\end{align}
\end{definition}
If $r < \infty$, the scheme $\{ \Psi^{(n)}_{} \}_{n \in \mathbb{N}} $ is said to have \emph{linear} resolution power.

The resolution constant pertains to the asymptotic ppw of a numerical scheme. In particular, we note that if $\{\Psi^{(n)}_{}\}_{n \in \mathbb{N}}$ is Chebyshev interpolation, then $r=\pi$, and if $\{\Psi^{(n)}_{}\}_{n \in \mathbb{N}}$ is Fourier interpolation, then $r = 2$ (provided in this case $\omega$ in~\eqref{eqn:resolutionConstant} is restricted to integer values). Although both of these schemes have linear resolution power, Fourier interpolation is clearly more efficient than Chebyshev interpolation for periodic functions.

At this stage, the reader may wonder why, in an article concerning functions with singularities, we analyse resolution by examining the analytic function $f(x) = e^{2 \pi i \omega x}$.  We note that our results regarding $\R$ are completely unchanged if we allow the more general form $f(x) = g(x) e^{2 \pi i \omega x}$, where $g(x)$ is analytic on $(0,1)$, continuous on $[0,1]$, and independent of $\omega$.  For simplicity, we consider the case $g(x) = 1$.

Let us also make the following remark.  Resolution power, although a quantitative way of comparing different methods, is not a substitute for a proper convergence analysis.  It does not, for example, provide any information about how well a method deals with other phenomena arising in function approximation, e.g.\ boundary layers.  For this reason,  we shall provide both analyses in this paper.  Namely, we derive error bounds for each of the methods considered and ppw estimates.

\subsection{Summary of results} In each of the semi-infinite and infinite interval settings, we will analyse an unparameterised map and a parameterised one, giving a total of four maps for consideration. To distinguish between the two situations, we introduce some notation: The symbol $\varphi$ denotes semi-infinite interval maps, whilst the symbol $\psi$ denotes infinite interval maps.  The symbols we use to denote the individual maps are $\phiE$, $\phiS$, $\psiE$, $\psiS$. The first two are semi-infinite interval transforms, and the second two infinite interval ones. The subscript ``E'' stands for ``exponential'', whilst ``S'' stands for ``slit-strip'' (see later).

For the paramaterised slit-strip maps $\phiS$ and $\psiS$, we derive both a convergence and a resolution result. For the unparamaterised maps $\phiE$ $\psiE$ we derive only a resolution estimate, since the convergence estimates were given already in~\cite{Richardson13}. A summary of our results is given in Table~\ref{table:resultsSummary}.  Note that here and elsewhere in the paper we use the notation $X(n) = \ordera(\chi(n))$ as $n \to \infty$ if there exists a fixed $p > 0$ such that $X(n) = \order( n^{p}_{} \chi(n))$ as $n \to \infty$.

Before we begin to establish these results, let us first explain how they come about.  By transplanting singular points to infinity, the convergence rate of the approximation is determined by a region of analyticity in the $s$-plane. Therefore we have two demands on that region.  First, it should contain as large a Bernstein ellipse as possible, and second it should make the image of this region in the $x$-plane -- which amounts to a condition that the original $f$ must satisfy -- as natural and accommodating to `typical' functions as possible.  As we shall see, the second is a singular failing of both $\phiE$ and $\psiE$, and is overcome by the new mappings $\phiS$ and $\psiS$.  The first requirement is met by taking $\alpha$ large and fixed.  On the other hand, resolution is determined by the amount of length distortion introduced by the conformal map, and how this interacts with the Chebyshev endpoint node clustering.  As we determine, good resolution requires a fixed $L$.  With the new maps this is possible, since one can vary $\alpha$ instead to obtain convergence.

\begin{table}[ht] 
\centering
\renewcommand{\arraystretch}{1.3} 
\begin{tabular}{c | c c c c}
                                  & $\phiE$, $\phiS$  & $\psiE$, $\psiS$ & $\phiS$ & $\psiS$
\\ \hline 
\multicolumn{1}{c|}{$\alpha$}  & ---, $\alpha^{}_{0}$ & ---, $\alpha^{}_{0}$ & $\alpha^{}_{0}/\sqrt{n}$ & $\alpha^{}_{0}/\sqrt{n}$  \\
\multicolumn{1}{c|}{$L$}  & $c n^{2/3}$ & $c \sqrt{n}$  & $1 + L^{}_{0}$ & $1/2 + L^{}_{0}$  \\
\multicolumn{1}{c|}{error}  & \rule{0pt}{0.5cm} $\ordera \s \left( C_{}^{-n^{2/3}} \right)$  & $\ordera \s \left( C_{}^{-\sqrt{n}} \right)$  & $\ordera \s \left( C_{}^{-\sqrt{n}} \right)$ & $\ordera \s \left( C_{}^{-\sqrt{n}} \right)$ \\ 
 \multicolumn{1}{c|}{d.o.f.}   & \rule{0pt}{0.5cm} $\order(\omega^{3/2})$ & $\order(\omega^2)$ & $\order(\omega)$ & $\order(\omega)$
\\ \multicolumn{1}{c|}{$r$}   & \rule{0pt}{0.5cm} $\infty$ & $ \infty$ & $(1+L^{}_{0})\pi$ & $(1 + 2L^{}_{0})\pi$
\end{tabular}\vspace{0.5em} 
\caption{\label{table:resultsSummary} \small 
A summary of the various convergence and resolution results. The rows indicate: the parameters $\alpha$ and $L$; the bounds on the approximation error; the degrees of freedom required to begin resolving  $\exp(2 \pi i \omega x)$ as $\omega \to \infty$; and the resolution constant $r$ (or the ppw figure). We analyse two parameter choice regimes: (i) $\alpha$ fixed, $L \to \infty$, summarised in the first two columns; and (ii) $L$ fixed, $\alpha \to 0$, summarised in the last two columns. Our key results are that with the new maps $\phiS$ and $\psiS$, one can obtain both rapid convergence and a finite ppw figure. We note that the figures for $r$ in the last two columns can be brought arbitrarily close to the optimal value of $\pi$ ppw by choosing $L^{}_0 > 0$ appropriately.
}
\end{table}

\section{Semi-infinite interval maps}\label{sect:semi} 

Let $f(x)$ be analytic on $(0,1]$ and continuous on $[0,1]$. Suppose further that $\varphi$ is an invertible mapping satisfying
\begin{align*}
 \varphi: (0,1] \mapsto (-\infty,0].
\end{align*}
We shall assume that $\varphi(0) = -\infty$, and $\varphi(1) = 0$.   $\F(s)$, the transplant of $f(x)$ to the $s$-variable, is defined thus:
\begin{align*}
 \F(s) = f(\varphi^{-1}_{}(s)), \qquad s \in  (-\infty,0].
\end{align*}

Our interest is in transforms $\varphi$ which represent an exponential change of variables. This means that, given mild smoothness assumptions on $f$, we wish to select a map $\varphi$ such that the corresponding transplant $\F(s)$ converges exponentially to the limiting value $f(0)$ as $s \to -\infty$. If this is the case, then rather than representing $\F$ on the interval $(-\infty,0]$, we instead construct an approximation on an interval $[-L,0]$, where $L>0$ is sufficiently large.  As discussed, we shall use Chebyshev interpolation for this purpose. Since the canonical Chebyshev interpolation domain is the unit interval, we define the following function, which represents a scaling of $[-L,0]$ to $[-1,1]$: 
\begin{align*}
 \FL(y) = \F(L(y-1)/2), \qquad y \in (-\infty,1].
\end{align*}
Note here that though, strictly speaking, $\FL$ is defined on a semi-infinite domain, we will only ever be interested in sampling it on the interval $[-1,1]$.

The three functions $f(x), \F(s)$ and $\FL(y)$ have a pointwise correspondence across their respective domains; see Figure~\ref{fig:functionsSemiInf}.

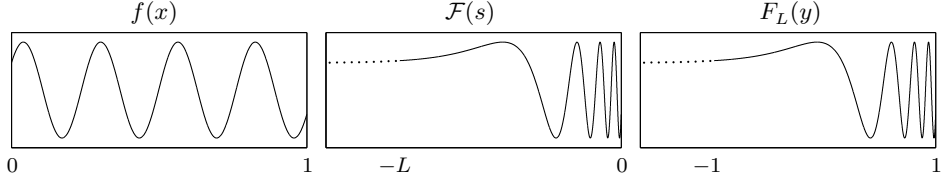
\begin{figure}[ht]
\begin{center}
\begin{overpic}[scale=0.68]{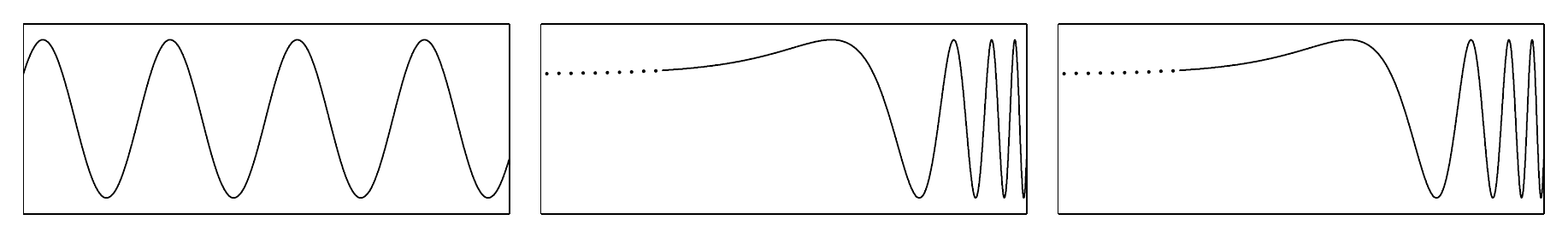}
\put(80,15){\small $\FL(y)$}
\put(47,15){\small $\F(s)$}
\put(14,15){\small $f(x)$}
\put(1,-1){\footnotesize $0$}
\put(32,-1){\footnotesize $1$}
\put(40,-1){\footnotesize $-L$}
\put(65,-1){\footnotesize $0$}
\put(73,-1){\footnotesize $-1$}
\put(98,-1){\footnotesize $1$}
\end{overpic}
\caption{\label{fig:functionsSemiInf} \small In order to approximate a given function $f(x)$ on $[0,1]$, we first transplant using the mapping $\varphi$ to a function $\F(s)$ defined on $(-\infty,0]$. Then we apply the domain-truncation strategy of scaling $\F(s)$ to a further function $\FL(y)$ in such a way that the interval $[-L,0]$ is mapped to $[-1,1]$. It is here that Chebyshev interpolation will be applied.}
\end{center}
\end{figure}

The function $\FL(y)$ will be approximated on $[-1,1]$ by the degree $n$ Chebyshev polynomial interpolant
\vspace{-1em}
\begin{align}\label{eqn:chebInterpolant}
 \Pn(y) = \sum_{k=0}^{n} c_k T_k(y).
\end{align}
As usual, $\{T_k\}$ are the first kind Chebyshev polynomials, and $c_k$ are aliased expansion coefficients corresponding to the function samples $\FL(y_j)$, where $y_j = \cos(j \pi/n)$, $j = 0,1,\ldots,n$ are the Chebyshev points of the second kind.

Since, by assumption, $f(x)$ is analytic on $(0,1]$, for any finite $L > 0$, the function $\FL(y)$ is analytic on $[-1,1]$ and therefore has an analytic continuation to a neighbourhood of the complex plane surrounding this interval. In such circumstances, the Chebyshev interpolants given by~\eqref{eqn:chebInterpolant} converge uniformly to $\FL$ geometrically as $n \to \infty$, at a rate governed by the size of the associated region of analyticity. 

More precisely, given some $\mu > 0$, we define the \emph{Bernstein ellipse}
\begin{align*}
 E_\mu = \left\{ \cos\theta\cosh\mu  +i\sin\theta\sinh\mu \, : \, \theta \in [0,2\pi) \right\}.
\end{align*}
This is an ellipse in the complex plane with foci $\pm 1$.
Whenever a function is analytic within such a region, we can apply the following result (see~\cite[Ch. 8]{TrefethenATAP}, for example):

\begin{theorem}[Bernstein]
 Let $\FL$ be analytic in the open region bounded by the ellipse $E_\mu$. Let the quantity
\begin{align}\label{eqn:ellipseMax}
 m(\mu;\FL) = \sup_{y \in E_\mu} |\FL(y)|
\end{align}
 be finite. Then, for every positive integer $n$, the Chebyshev interpolant~\eqref{eqn:chebInterpolant} satisfies
\begin{align}\label{eqn:BernsteinBound}
 \| \FL - \Pn \|^{}_{y \in [-1,1]} \leq \frac{4}{\mu} m(\mu ;\FL) e^{-\mu n}.
\end{align}
\end{theorem}
The pointwise correspondence between $\FL$ and $f$ means that $\Pn$ does not provide an approximation to $f$ over the whole of $[0,1]$, but rather only over a sub-interval, $[\xL,1]$, where $\xL = \varphi^{-1}_{}(-L)$. To form an approximation over the entirety of $[0,1]$, it is necessary to additionally prescribe an approximation over the remaining region $[0,\xL)$. To this end, we define our approximation to $f$ on $[0,1]$ as follows:
\begin{align}\label{eqn:pnLphi}
 \pnL(x) = \left\{\begin{array}{cl}
  \FL(-1),                    & \quad    x \in [0,\xL),   \\
  \Pn(2\varphi(x)/L+1 ),   & \quad    x \in [\xL,1].      
\end{array}\right.
\end{align}
This is a piecewise representation. The $\infty$-norm error is the maximum of the errors over each piece, and therefore we have
\begin{align}\label{eqn:pnLphiError}    
 \| f-\pnL \| = \max\big\{ \, \|\FL - \Pn \|^{}_{y \in [-1,1]} \, , 
	   \, \| \FL-\FL(-1) \|^{}_{y \in (-\infty,-1)} \, \big\}.
\end{align}

\subsection{Unparameterised exponential map \boldmath{$\phiE$}}

Consider first the simplest possible exponential transform and its corresponding inverse:
\begin{align*}
 \phiE(x) = \log(x), \qquad \phiEi(s) = \exp(s) .
\end{align*}

\subsubsection{Convergence rate for \boldmath{$\phiE$}} We assume analyticity of $\F(s)$ in the region
\begin{align}\label{eqn:parabolicRegion}
 \P_{d}^{} = \left\{ z \in \complex : \Re(z) < \frac{d}{2} - \frac{1}{2 d} \Im(z)^2 \right\},
\end{align} 
where $d$ is a positive number. This is the open region in the complex plane bounded by the parabola with focus $z = 0$ and directrix $\Re(z) = d$. The image of $\P_{d}^{}$ under the exponential map is an infinitely sheeted Riemann surface wrapping around at $x=0$. Figure~\ref{fig:parabolicRegion} provides an illustration of these regions. 

\begin{figure}[ht]
\begin{center}
\begin{overpic}[scale=0.45]{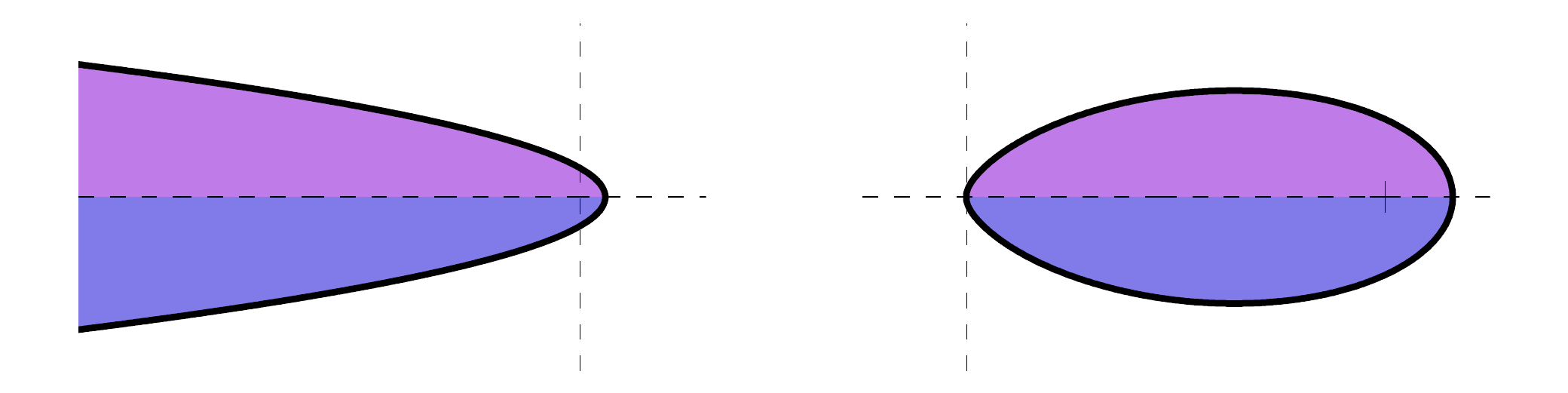}
\put(18,1){\footnotesize $s$-plane}
\put(74,1){\footnotesize $x$-plane}
\put(18,22){\footnotesize $\P_{d}^{}$}
\put(73.5,22){\footnotesize $\phiEi(\P_{d}^{})$}
\end{overpic}
\caption{\label{fig:parabolicRegion} \small The parabola $\P_{d}^{}$ and its image under the exponential map, $\phiEi(\P_{d}^{})$. The region on the right is an infinitely sheeted Riemann surface wrapping around at $x = 0$.}
\end{center}
\end{figure}
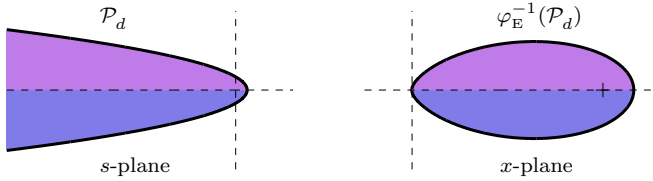

\begin{theorem}[\cite{Richardson13}, Thm.~3.4]\label{thm:phiE} Let $f(x)$ be analytic and bounded on the Riemann surface $\phiEi(\P_{d}^{})$ for some $d > 0$. Let $f$ satisfy $|f(x)-f(0)| = \order(|x|^\tau)$, as $x \to 0$ for some $\tau > 0$. Let $\pnL$ be the approximation defined by~\eqref{eqn:pnLphi} corresponding to the map $\varphi = \phiE$. If $c>0$ and  $L = c n^{2/3}$, then as $n \to \infty$,
   \begin{align}\label{thm:phiResult} 
    \|f - \pnL\| = \ordera \s \left(C^{-n^{2/3}}\right) , \qquad C = \min\left\{ \exp\left( \sqrt{2d/c} \right) , \exp\left( \tau c \right) \right\}.
   \end{align}
\end{theorem}
Note that the choice $L = c n_{}^{2/3}$ in this theorem is made to ensure that the two terms in the error expression~\eqref{eqn:pnLphiError} decay at the same rate.  Allowing $L$ to scale either faster or slower than $n^{2/3}$ as $n \rightarrow \infty$ would lead to slower exponential convergence.

\subsubsection{Resolution analysis of $\phiE$}

Our strategy for analysing the resolution properties of $\phiE$ and the other maps in this paper is as follows. We first estimate the quantity~\eqref{eqn:ellipseMax} by analysing the maximal behaviour of the function $f(x) = e_{}^{2\pi i \omega x}$ on the Bernstein ellipse $E^{}_{\mu}$ for large $| \omega |$.  We then use this to determine the smallest value of $n$ for which exponential convergence of the right-hand side of~\eqref{eqn:BernsteinBound} occurs.

To proceed, we set $f(x) = e^{2\pi i \omega x}$  and consider the approximation to $f$ defined by~\eqref{eqn:pnLphi}, where $\varphi = \phiE$. The function $\FL$ in this case is given by
\begin{align*}
 \FL(y) 
 = \exp\left( 2 \pi i \omega \phiEi\left( \frac{L}{2} (y-1) \right) \right)
 = \exp\left( 2 \pi i \omega \exp\left( \frac{L}{2} (y-1) \right) \right).
\end{align*}
Setting $y = \cos\theta\cosh\mu  +i\sin\theta\sinh\mu$, for $\theta \in [0,2\pi)$, we obtain
\begin{align*}
\FL(y) 
&= \exp\left( 2 \pi i \omega \exp\left( \frac{L}{2} (\cos\theta\cosh\mu-1) \right) \exp\left( i \frac{L}{2} \sin\theta\sinh\mu \right) \right),
\end{align*}
from which it follows that 
\begin{align*}
 |\FL(y)| = \exp(-2\pi \omega H(\theta,\mu,L)), 
\end{align*}
where
\vspace{-1em}

\begin{align*}
 H(\theta,\mu, L) = \exp\left( \frac{L}{2} (\cos\theta\cosh\mu-1) \right) \sin\left( \frac{L}{2} \sin\theta\sinh\mu \right).
\end{align*}
Thus, given $\mu$ and $L$, it is clear that that maximising the quantity $|\FL(y)|$ over $E_\mu$ corresponds to minimising the function $H(\theta,\mu,L)$ over $\theta \in [0,2\pi)$.

Our interest lies in the asymptotic regime $\omega \rightarrow \infty$.  As $\omega$ increases, one would expect the parameter $n$ also to increase, since more degrees of freedom are generally required to resolve oscillatory functions.  Since the exponential map requires large values of $L$ to achieve good accuracy, it is therefore reasonable to consider asymptotic expansions in $L \rightarrow \infty$.  Since we are free to choose $\mu > 0$ in~\eqref{eqn:BernsteinBound}, we shall simultaneously let $\mu \rightarrow 0$ as $L \rightarrow \infty$ in such a way that $L \mu \rightarrow 0$. Doing this gives
\begin{align*}
 H(\theta,\mu,L) \sim \exp\left( \frac{L}{2} (\cos\theta-1)\right) \frac{1}{2} L \mu \sin \theta = \frac{1}{2} L \mu K(\theta),
\end{align*}
where
\vspace{-1em}
\begin{align*}
 K(\theta) = \exp\left( \frac{L}{2} (\cos\theta-1)\right) \sin \theta.
\end{align*}
It follows that
\vspace{-1em}
\begin{align*}
 \frac{{\rm \partial}}{{\rm \partial} \theta} H(\theta,\mu,L) \sim 0 \quad \Leftrightarrow \quad \frac{{\rm d}}{{\rm d} \theta} K(\theta) = 0.  
\end{align*}
This last equation can be solved explicitly for $\theta$ in terms of $L$ to give 
\begin{align*}
 \theta = \cos^{-1} \left(\frac{-1 + \sqrt{1 + L^2} }{ L } \right).
\end{align*}
There are precisely two permissible values of $\theta$ in the interval $[0,2\pi)$: One lies in $[0,\pi)$, the other in $[\pi,2\pi)$. Let the first of these be $\theta^+$. Then, since $H(\theta,\mu, L)$ is antisymmetric about $\theta = \pi$, satisfying $H(\eta,\mu, L) = - H(2\pi-\eta,\mu, L)$ for $\eta \in [0,\pi)$, the other is $\theta^* = 2\pi - \theta^+$. One can easily check that $\theta^+$ corresponds to maximum of $K(\theta)$ and $\theta^*$ to the minimum. One can also easily show that
\begin{align*}
 \theta^* \sim 2\pi - \sqrt{\frac{2}{L}}, \qquad L \to \infty,
\end{align*}
so that for all $\theta \in [0,2\pi)$ we have
\begin{align*}
 K(\theta) \geq K(\theta^*) \sim - \sqrt{\frac{2}{e L}},\qquad L \to \infty.
\end{align*}
Thus,~\eqref{eqn:ellipseMax} reads
\vspace{-1em}
\begin{align*}
 m(\mu;\FL) \sim \exp\left( \mu \omega \sqrt{\frac{2 \pi^2 L}{e}} \right), \qquad L \to \infty.
\end{align*}
Combining this together with~\eqref{eqn:BernsteinBound}, as $\mu \to 0$, $L \to \infty$, $L \mu \to 0$, we have
\begin{align*}
  \| \FL - \Pn \| 
    \leq \frac{4}{\mu} \exp\left( \mu \omega \sqrt{\frac{2 \pi^2 L}{e}}  - \mu n \right).
\end{align*}
A sufficient condition for convergence is therefore  
\begin{align*}
 n \geq  \sqrt{ \frac{2  L}{e}} \pi \omega.
\end{align*}

Lastly, in order to obtain the $\ordera(C^{-n^{2/3}})$ convergence rate, we set $L =cn^{2/3}$, as prescribed in Theorem~\ref{thm:phiE}. We have thus proven the following result:

\begin{theorem}
Let $\R(\omega;\delta)$ be given by~\eqref{eqn:epsilonResolution} where $\{\Psi^{(n)}\}_{n \in \natural}$ is the approximation scheme defined by~\eqref{eqn:pnLphi} corresponding to the map $\varphi = \phiE$. Let $L = cn^{2/3}$ for any $c > 0$. Then, 
\vspace{-1em}
\begin{align}\label{eqn:resolutionFigure_phiE}
 \limsup_{\delta \to 1^{-}} \limsup_{|\omega| \to \infty} \frac{\R(\omega;\delta)}{|\omega|^\frac{3}{2}} \leq \pi^\frac{3}{2} \left(\frac{2 c}{e}\right)^\frac{3}{4}.
\end{align}
\end{theorem}
A numerical verification of this result is given in Figure~\ref{fig:numericalExperiments_phiE}. We note that the bound~\eqref{eqn:resolutionFigure_phiE} appears to be sharp in practice.

\begin{figure}[ht]
\begin{center}
\begin{overpic}[scale=0.55]{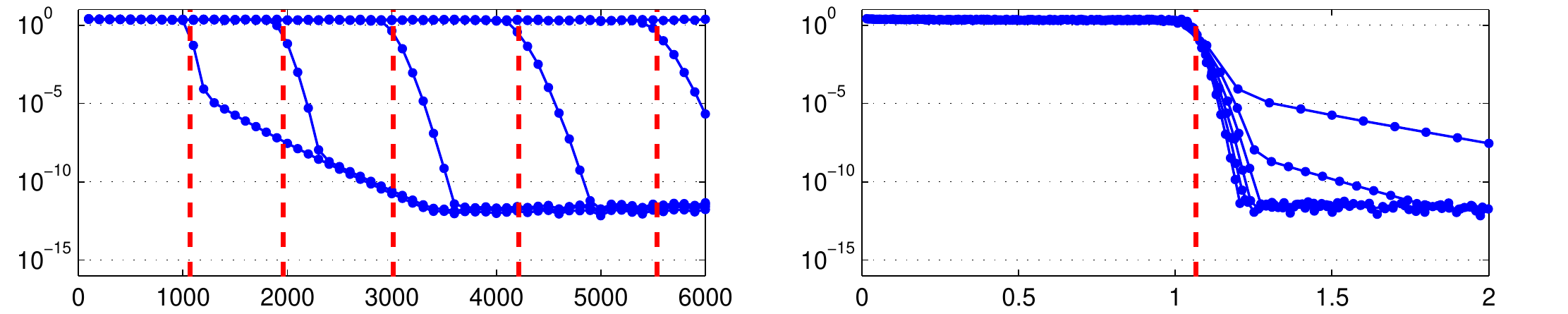}
 \put(25,-3) {\footnotesize $n$}
 \put(73,-3) {\footnotesize $n/\omega^\frac{3}{2}$}
\end{overpic}
\vspace{0.5em}
\caption{\label{fig:numericalExperiments_phiE} \small The error $\|f - \pnL\|$ in approximating $f(x) = e^{2 \pi i \omega x}$ on $[0,1]$ using~\eqref{eqn:pnLphi} in conjunction with the map $\varphi = \phiE$ for $\omega = 100, 150, \ldots, 350$. In each case, we use $c = 0.15$ in the relationship $L = c n^{2/3}$. In the left pane, for each $\omega$, we display the error against $n$. In the right pane, we show the same data, but with the $x$-axis rescaled by a factor of $\omega^{3/2}$. The dashed red lines indicate the theoretical resolution figure ($\approx1.066 \omega^{3/2}$) given by~\eqref{eqn:resolutionFigure_phiE}  corresponding to this particular choice of $c$. }
\end{center}
\end{figure}

In summary, the numerical method based upon the transform $\phiE$ converges at the rate $\ordera(C^{-n^{2/3}})$, requiring $\order(\sqrt{\omega})$ points per wavelength, and therefore $\order(\omega^{3/2})$ points in total, to resolve the complex exponential $e^{2 \pi i \omega x}$.

Given~\eqref{eqn:resolutionFigure_phiE}, one may be led to conclude that $c$ should be set as small as possible so as to mitigate the effect of suboptimal resolution power. However, whilst doing this will indeed improve the resolution, it will also worsen the convergence. Indeed, from~\eqref{thm:phiResult} it is clear that $C \sim 1 + \order(c)$ as $c \to 0$.

\subsection{Parameterised exponential map $\phiS$} So far, we have seen that numerical methods based upon the map $\phiE$ lead to approximations which require $\order(\omega^{3/2})$ points to begin resolving resolve the complex exponential $e^{2\pi i \omega x}$. In this section, we shall derive a new transformation involving a user-specified parameter $\alpha$ and show that it is possible to choose this parameter in such a way that the resulting numerical method requires just  $\order(\omega)$ points to resolve $e^{2\pi i \omega x}$.

We denote by $\S_\alpha^{}$ the open infinite strip of half-width $\alpha$ defined by
\begin{align}\label{eqn:infiniteStrip}
 \S_\alpha^{} = \left\{ z \in \complex : |\Im(z)| < \alpha \right\}.
\end{align}
For comparative purposes, consider the action of the unparameterised exponential map $\phiEi$ upon this strip. For any $0 < \alpha < \pi$, $\phiEi$ is an analytic function throughout $\S_\alpha^{}$, mapping it in a one-to-one fashion onto the wedge-shaped region formed by two straight lines meeting with half-angle $\alpha$ at the origin; see Figure~\ref{fig:stripToWedge}. Our derivation of the new map $\phiSi$ is motivated by the notion that a wedge-shaped region is undesirable since it requires ``more analyticity'' of the function near $x=1$ than near $x=0$.

\begin{figure}[ht]
\begin{center}
\begin{overpic}[scale=0.45]{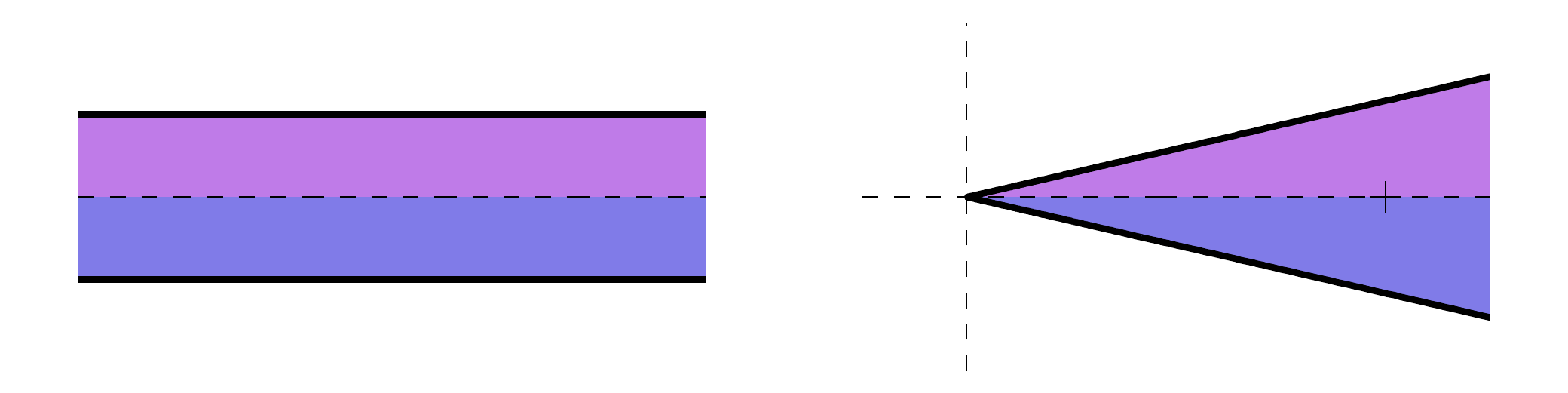}
\put(18,1){\footnotesize $s$-plane}
\put(73,1){\footnotesize $x$-plane}
\put(87.8,8.5){\footnotesize $1$}
\put(21,20){\footnotesize $\S_\alpha^{}$}
\put(73,20){\footnotesize $\phiEi(\S_\alpha^{})$}
\put(1,17.2){\footnotesize $i\alpha$}
\put(-1.5,6){\footnotesize $-i\alpha$}
\end{overpic}
\caption{\label{fig:stripToWedge} \small For any $\alpha < \pi$, $\phiEi$ maps the infinite strip of half-width $\alpha$ onto the wedge-shaped region formed by two straight lines meeting with half-angle $\alpha$ at $x=0$.}
\end{center}
\end{figure}
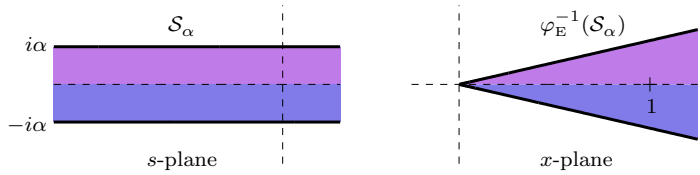

We intend to construct a transform which maps $\S_\alpha^{}$ onto an infinite strip of the same width, but with a slit on the negative real axis, $(-\infty,0]$. We refer to this region as a \emph{slit strip of half-width $\alpha$}; see the top-right pane of Figure~\ref{fig:slitToSlitStrip}. 

This can be achieved as follows. First, we use the map
\begin{align*}
 \varphi_1^{-1}(s;\alpha)  = i \exp\left( \frac{\pi s}{2\alpha} \right),
\end{align*}
to take us from $\S^{}_{\alpha}$ to the half-plane. 
Next, we utilise the Schwarz-Christoffel formula to construct a map $\varphi_2^{-1}(t)$ which takes us from the half-plane to the slit-strip. For more on Schwarz-Christoffel mapping, see~\cite{SCmapping}; for some related articles containing similar themes to this part of our investigation, see~\cite{HaleTee09,HaleTrefethen08,HowellTrefethen90}. 

One can view the destination region in the $x$-variable as a polygon with one finite and three infinite vertices. The two steps described above are illustrated in Figure~\ref{fig:slitToSlitStrip}.

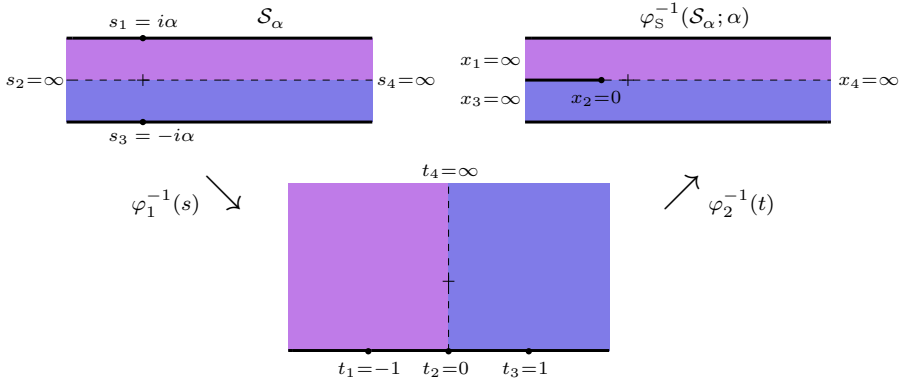
\begin{figure}[ht]
\begin{center}
\begin{overpic}[scale=0.45]{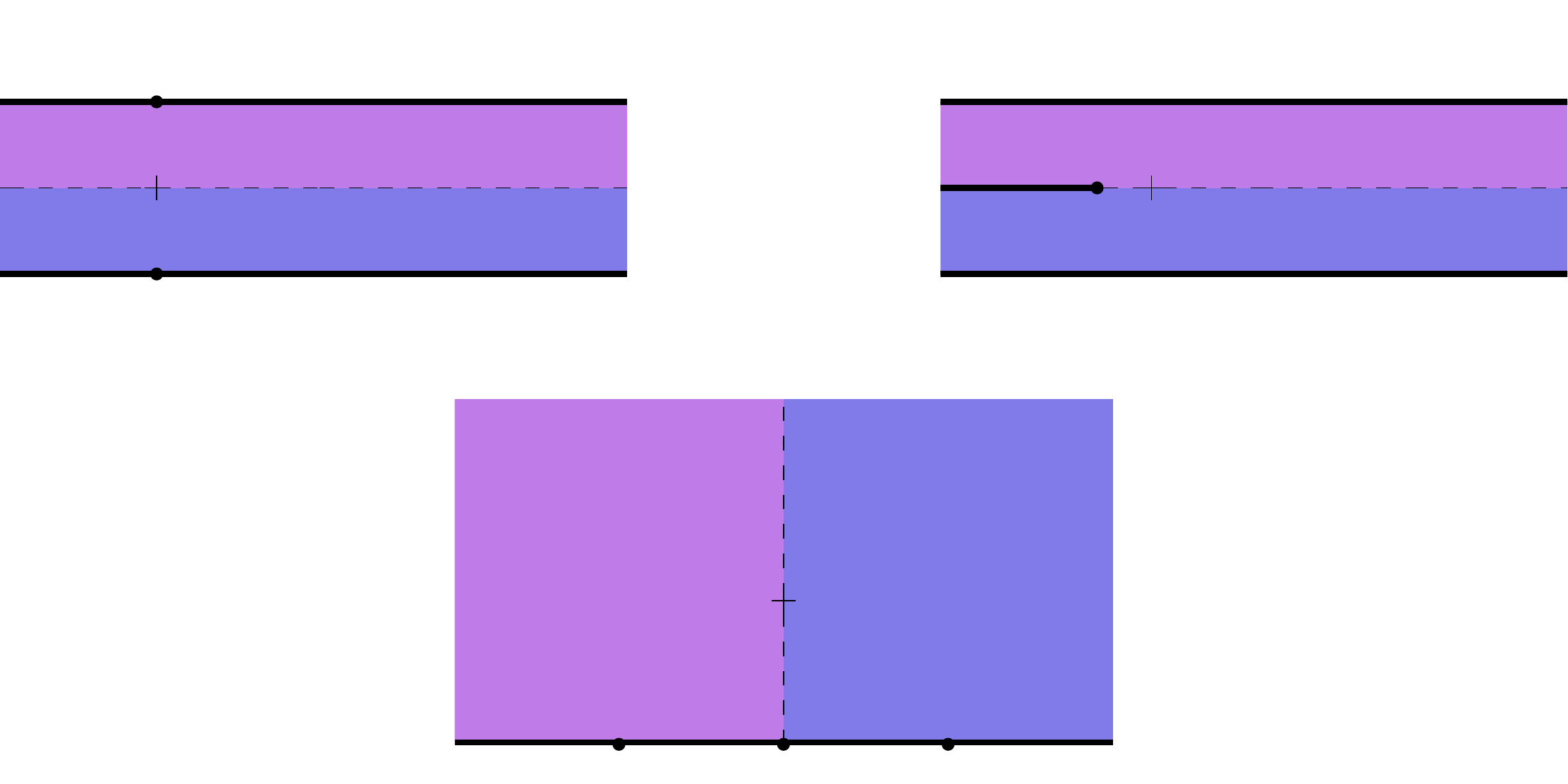}
 \put(18,22) {\Large $\searrow$}
 \put(8.5,21) {\footnotesize $\varphi_1^{-1}(s)$}
 \put(78,22) {\Large $\nearrow$}
 \put(84,21) {\footnotesize $\varphi_2^{-1}(t)$}
 \put(5.5,45) {\scriptsize $s_1 = i \alpha$}
 \put(-8,37.3) {\scriptsize $s_2 \ee \infty$}
 \put(5.5,30) {\scriptsize $s_3 = -i \alpha$}
 \put(40.6,37.3) {\scriptsize $s_4 \ee \infty$}
 \put(25,45.5) {\footnotesize $\S^{}_\alpha$}
 \put(35.5,-0.5) {\scriptsize $t_1 \ee -1$}
 \put(46.6,-0.5) {\scriptsize $t_2 \ee 0$}
 \put(57,-0.5) {\scriptsize $t_3 \ee 1$}
 \put(46.5,25.5) {\scriptsize $t_4 \ee \infty$}
 \put(51.5,40) {\scriptsize $x_1 \ee \infty$}
 \put(66,35) {\scriptsize $x_2 \ee 0$}
 \put(51.5,34.9) {\scriptsize $x_3 \ee \infty$}
 \put(101,37.3) {\scriptsize $x_4 \ee \infty$}
 \put(75,45.5) {\footnotesize $\phiSi(\S^{}_\alpha;\alpha$)}
\end{overpic}
\caption{\label{fig:slitToSlitStrip} \small Construction
via Schwarz-Christoffel mapping of the transformation 
$\phiSi$ which maps the infinite strip of half-width $\alpha$ 
to the infinite slit-strip of half-width $\alpha$. Note the uniform analyticity requirement of $\phiSi(\S^{}_\alpha;\alpha)$ in contrast to the region $\phiEi(\S^{}_{\alpha})$ from Figure~\ref{fig:stripToWedge}.}
\end{center}
\end{figure}
For the Schwarz-Christoffel step, the prevertices $t_j$, vertices $x_j$ and angles $\delta_j$ are
\begin{table}[ht] 
\vspace{-0.5em} 
\centering 
\renewcommand{\arraystretch}{1} 
\setlength{\tabcolsep}{15pt}
\begin{tabular}{lll}
$t_1 = -1$,                   & $x_1 = \infty$,     &  $\delta_1 = 0$, \\
$t_2 = 0$,                    & $x_2 = 0$,          &  $\delta_2 = 2$, \\
$t_3 = 1$,                    & $x_3 = \infty$,     &  $\delta_3 = 0$, \\
$t_4 = \infty$,               & $x_4 = \infty$,     &  $\delta_4 = 0$.   
\end{tabular}
\vspace{-0.5em} 
\end{table}
\noindent

The Schwarz-Christoffel integral can be evaluated exactly. We have
\begin{align*}
 \varphi_2^{-1}(t)
 = \tilde{B} + \tilde{C} \int^{t} (\xi + 1)^{-1}\xi(\xi - 1)^{-1} {\rm d} \xi
 = \tilde{B} + C \log(t^2-1),
\end{align*}
for some constants $\tilde{B}, \tilde{C}, C$. Composing the two maps gives 
\begin{align*}
 \tilde{\varphi}^{-1}_\text{\tiny S}(s;\alpha) = \varphi_2^{-1} \circ \varphi_1^{-1}(s;\alpha) 
    = B + C \log\left(1 + e^{\pi s/\alpha}\right),
\end{align*}
for some $B$. To determine the constants, we enforce the two conditions
\begin{align*}
 \mbox{i) \,}  \lim_{s \to -\infty} \tilde{\varphi}^{-1}_\text{\tiny S}(s;\alpha) = 0, 
 \qquad \mbox{ii) \,} \Im(\tilde{\varphi}^{-1}_\text{\tiny S}(c + i \alpha;\alpha)) = \alpha,
\end{align*}
where $c \in \real$. This gives $B = 0$ and $C = \alpha/\pi$, respectively. Thus, we have
\begin{align*}
 \tilde{\varphi}_\text{\tiny S}(x;\alpha) = \frac{\alpha}{\pi} \log\left( e^{\pi x/\alpha} - 1 \right), \qquad \tilde{\varphi}^{-1}_\text{\tiny S}(s;\alpha) = \frac{\alpha}{\pi} \log\left( 1 + e^{\pi s/\alpha} \right).
\end{align*}
Finally, we shift these functions by the quantity
\begin{align}\label{eqn:gamma}
 \gamma = \tilde{\varphi}_\text{\tiny S}(1;\alpha) = \frac{\alpha}{\pi} \log\left( e^{\pi/\alpha} - 1 \right), 
\end{align}
in order to obtain
\begin{align*}
 \phiS(x;\alpha) = \frac{\alpha}{\pi} \log\left( e^{\pi x/\alpha} - 1 \right) - \gamma, \quad \phiSi(s;\alpha) = \frac{\alpha}{\pi} \log\left( 1 + e^{\pi (s+\gamma)/\alpha} \right).
\end{align*}
This shift ensures that, for any $\alpha > 0$, $\phiS(\,\cdot\,;\alpha): (0,1] \mapsto (-\infty,0]$, which is a property consistent with the action of the map $\phiE(x) = \log(x)$. We note that $\gamma \to 1$ as $\alpha \to 0$.

\subsection{Convergence analysis of $\phiS$}\label{ss:conv_phiS}
As detailed in Theorem~\ref{thm:phiE}, the numerical method based on $\phiE$ involves just a single user-specified \emph{domain-truncation} parameter $L$. However, the numerical method based on $\phiS$ involves a further \emph{strip-width} parameter $\alpha$, in addition to $L$. A consequence of this, as we shall now see, is that it is possible to obtain convergence of our approximations by two different approaches.  The first is to fix $\alpha$, and as we did with the map $\phiE$, allow $L$ to increase as a function of $n$.  However, this also leads to the poor resolution properties associated with $\phiE$.  The second approach involves keeping $L$ bounded and letting $\alpha$ tend to zero as $n$ increases. As we shall show, this leads to vastly superior resolution properties.  Specifically, $\order(\omega)$ points are required to resolve $e^{2\pi i \omega x}$, as opposed to $\order(\omega^{3/2})$.

The following result describes convergence in the first of these situations:

\begin{theorem}\label{thm:phiS_alphaFixed} Let $f$ be analytic and bounded on the Riemann surface $\phiSi(\P_{d}^{};\alpha)$ where $\alpha > 0$ is fixed, and $0 < d \leq -\gamma + \sqrt{\gamma^2 + \alpha^2}$. Let $f$ satisfy $|f(x)-f(0)| = \order(|x|^\tau)$ as $x \to 0$ for some $\tau > 0$. Let $\pnL$ denote the approximation defined by~\eqref{eqn:pnLphi} corresponding to the map $\varphi = \phiS(\,\cdot\,;\alpha)$. Then for any $c > 0$, the choice $L = cn^{2/3}$ minimises the rate of decay of the error as $n \to \infty$, in which case we have
\begin{align}\label{thm:phiS_alphaFixedResult} 
    \|f - \pnL\| = \ordera \s \left(C^{-n^{2/3}}\right) ,\qquad C = \min \left \{ \exp \left ( \s \sqrt{ 2d/c } \, \right ) , \exp \left ( \frac{\pi \tau c}{\alpha} \right ) \right \}.
   \end{align}
\end{theorem}

\begin{proof} We first claim that if $L \to \infty$ as $n \to \infty$, then
\begin{align}\label{thm:phiS_alphaFixedResult1} 
    \|f - \pnL\| = \order \left( \max \left \{ \sqrt{L}\exp\left(- \sqrt{2d/L} \, n \right) , \exp\left(\frac{\pi \tau}{\alpha}(\gamma-L)\right) \right \} \right).
\end{align}
This can be seen as follows. First we let $\s s^{}_{1} + is^{}_{2}$ denote the location of the limiting singularity of the transplant $\F(s;\alpha) = f(\phiSi(s;\alpha))$. Since this is assumed to lie on the boundary of $\P^{}_d$, we see using~\eqref{eqn:parabolicRegion} that $d = s^{}_{1} + \sqrt{s^{2}_{1}+s^{2}_{2}}$.
Consider next the convergence rate of Chebyshev interpolants to the the function  $\FL(y;\alpha) = \F(L(y-1)/2;\alpha)$. This can be determined by setting
\begin{displaymath}
\cos \theta \cosh \mu + i \sin \theta \sinh \mu = y^{}_{1} + i y^{}_{2},
\end{displaymath}
where $y^{}_{1} = 2s^{}_{1}/L + 1$ and $y^{}_{2} = 2s^{}_{2}/L$. Upon eliminating $\theta$, one finds that 
\begin{equation}\label{eqn:muParabolaEllipse}
 \sinh^{2}_{} \mu =  \frac{1}{2} \left( y_{1}^{2} + y_{2}^{2} - 1 +\sqrt{ (1 - y_{1}^{2} - y_{2}^{2})^2 + 4y_{2}^{2} } \right).
\end{equation}
A straightforward calculation then shows that, as $L \to \infty$, 
\begin{align*}
 \mu \sim \sqrt{2d/L}.
\end{align*}
Using~\eqref{eqn:BernsteinBound}, it therefore follows that 
\begin{align*}
 \| \FL - \pnL \|^{}_{y \in [-1,1]} = \order\left( \sqrt{L} \exp\left( \sqrt{2d/L} n \right) \right), \qquad L \to \infty.
\end{align*}
This gives the first error term in~\eqref{thm:phiS_alphaFixedResult1}. To obtain the second term, we observe that
\begin{align*}
\xL = \phiSi(-L) = \frac{\alpha}{\pi} \log \left ( 1 + e^{\pi(-L+\gamma)/\alpha} \right ) \sim \frac{\alpha}{\pi} e^{\pi (-L+\gamma)/\alpha},\qquad L \rightarrow \infty.
\end{align*}
combining this together with our assumptions on $f$ gives 
$$
\| \FL - \FL(-1) \|^{}_{y \in (-\infty,-1)} = \order\left( \exp\left(\frac{\pi \tau}{\alpha}(\gamma-L)\right) \right).
$$
This gives us the second term in~\eqref{thm:phiS_alphaFixedResult1}. To obtain the main result of the theorem, we notice first that the exponentials in~\eqref{thm:phiS_alphaFixedResult1} decay at the same rate as $n \to \infty$ if and only if $L \sim c n^{2/3}$. Therefore we set $L = c n^{2/3}$ for any $c > 0$, which gives~\eqref{thm:phiS_alphaFixedResult}. 
\end{proof}

Theorem~\ref{thm:phiS_alphaFixed} shows that the optimal rate of convergence when $L$ is increased with $n$ is the same as that corresponding to the exponential map $\phiE$.

We consider next the alternative case in which $L$ is bounded and $\alpha \to 0$ as $n \to \infty$. 
\begin{theorem} 
\label{thm:phiS_LFixed} Let $f$ be analytic in the slit-strip $\S^{}_{\beta}\backslash(-\infty,0]$ for some $\beta > 0$. Let $f$ satisfy $|f(x)-f(0)| = \order(|x|^{\tau}_{})$ as $x \to 0$ for some $\tau > 0$. Let $\pnL$ be the approximation defined by~\eqref{eqn:pnLphi} corresponding to the map $\varphi = \phiS(\,\cdot\,;\alpha)$. If $L$ is bounded for all $n$, and $\alpha, \frac{\alpha}{L-1} \to 0$ as $n \to \infty$, then 
\begin{equation}\label{thm:phiS_LFixedResult1}
\| f - p^{(n)}_L \| = \order \left( \max \left\{ \frac{\sqrt{L-1}}{\alpha} \exp \left ( - \frac{\alpha n}{\sqrt{L-1}} \right )  , \exp \left ( \frac{\pi \tau}{\alpha} (1-L) \right) \right\} \right),
\end{equation}
as $n \rightarrow \infty$.  In particular, the error is minimised when $\alpha = \alpha_0 / \sqrt{n}$ and $L = 1 + L_0$ for some $\alpha_0,L_0>0$, in which case
\begin{equation}\label{thm:phiS_LFixedResult2}
\begin{split}
 \| f - p^{(n)}_L \| = \ordera \left ( C_{}^{-\sqrt{n}} \right ) , \quad C = \min \left \{ \exp\left( \frac{\alpha_0}{\sqrt{L_0}} \right) , \exp\left( \frac{\pi \tau L_0}{\alpha_0} \right) \right \}.
\end{split}
\end{equation}

\end{theorem}

\begin{proof} We first note that for all $\alpha \leq \beta$, the function $f$ is analytic in $\S_{\alpha} \backslash (-\infty,0]$. Therefore for all sufficiently small $\alpha$, the transplant $\F(s;\alpha) = f(\phiSi(s;\alpha))$ is analytic in the strip $\S^{}_\alpha$.  Correspondingly, $\FL(y;\alpha)$ is analytic in the strip $\S^{}_{\alpha/L}$. Specifically, $\phiSi$ introduces limiting singularities in $\FL(y;\alpha)$ at $1-2 \gamma / L \pm 2 \alpha / L i$. The precise rate of convergence in this regime may therefore be determined by setting $y^{}_{1} = 1 - 2\gamma/L + 1$ and $y^{}_{2} = 2\alpha/L$ in \eqref{eqn:muParabolaEllipse}. (Recall $\gamma$ is given by~\eqref{eqn:gamma}.) Doing this, we find  
\begin{align*}
 \mu \sim \alpha / \sqrt{L - 1}, \qquad \alpha, \alpha/(L-1) \to 0.
\end{align*}
Subsitituting this into~\eqref{eqn:BernsteinBound} then gives us the first error term in~\eqref{thm:phiS_LFixedResult1}. 

Consider next the domain error. As $\alpha, \alpha/(L-1) \rightarrow 0$ we see that 
\begin{align*}
\xL = \frac{\alpha}{\pi}\log \left ( 1 + e^{\pi (-L+\gamma)/\alpha} \right ) \sim \frac{\alpha}{\pi} \log \left ( 1 + e^{-\pi (L-1)/\alpha} \right ) \sim \frac{\alpha}{\pi} e^{-\pi (L-1)/\alpha}.
\end{align*}
This gives us the remaining term in~\eqref{thm:phiS_LFixedResult1}.

The second part of the result~\eqref{thm:phiS_LFixedResult2} is obtained by noting that both error terms in~\eqref{thm:phiS_LFixedResult1} decay at the same rate if, as $n \to \infty$, for some $\nu >0$, we have
\begin{align}\label{step}
\frac{L-1}{\alpha} \sim \nu \frac{\alpha}{\sqrt{L-1}} n \quad \Leftrightarrow \quad L-1 \sim \nu \alpha^{4/3} n^{2/3}.
\end{align}
The corresponding rate of convergence is therefore $\ordera(C^{-\alpha^{1/3} n^{2/3}})$. Since we require $\alpha \to 0$, it follows that we seek the slowest decay of $\alpha$ permitted. Moreover, because $L$ is assumed to be bounded,~\eqref{step} shows that the optimal choice is $\alpha = \alpha_0 / \sqrt{n}$, in which case we have $L = 1+L_0$ with $\alpha_0^{}, L^{}_0>0$ fixed. This gives~\eqref{thm:phiS_LFixedResult2}.
\end{proof}

This theorem shows that convergence of $\pnL$ is possible when $L$ is bounded, but that in such cases the convergence necessarily declines to root-exponential in $n$.   Figure~\ref{fig:convOneSided} provides numerical experiments verifying Theorems~\ref{thm:phiS_alphaFixed} and~\ref{thm:phiS_LFixed}. 

\begin{figure}[ht]
\begin{center}
\begin{overpic}[scale=0.7]{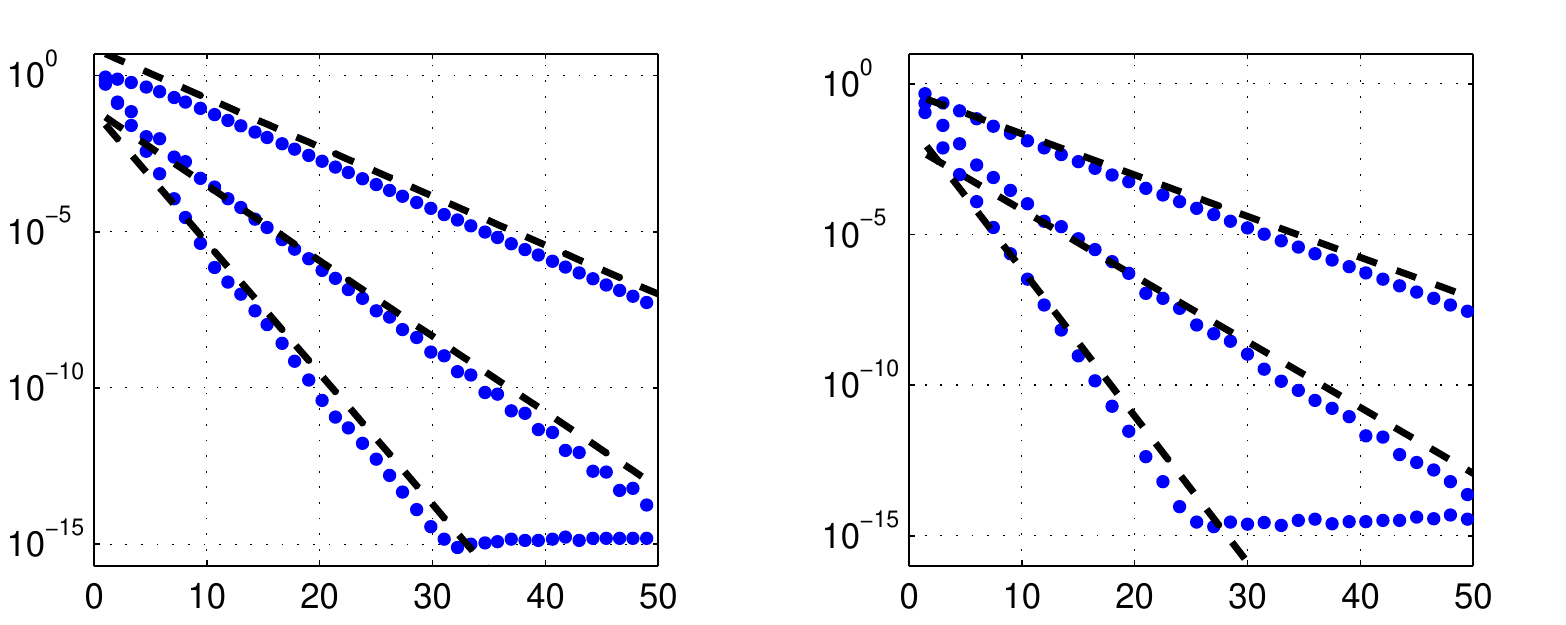}
 \put(13,39){\small $\alpha$ fixed, $L = cn^{2/3}$}
 \put(64,39){\small $L$ fixed, $\alpha = \alpha^{}_0/\sqrt{n}$}
 \put(21,-4){\small $n^{2/3}$}
 \put(39,12){\small $c = 0.23$}
 \put(14,8){\small $c = 0.9$}
 \put(32,27){\small $c = 2.7$}
 \put(74,-4){\small $\sqrt{n}$}
 \put(90,12.5){\small $\alpha^{}_{0} = 0.45$}
 \put(64,7){\small $\alpha^{}_{0} = 1$}
 \put(83,26){\small $\alpha^{}_{0} = 4$}
\end{overpic}
\vspace{0.5em}
\caption{\label{fig:convOneSided} \small Numerical verification of Theorem~\ref{thm:phiS_alphaFixed} (left) and Theorem~\ref{thm:phiS_LFixed} (right) for approximation of the function $f(x) = \sqrt{x}$ on $[0,1]$ using numerical methods based upon the map $\phiS(\,\cdot\,;\alpha)$. In the left pane, we have used the value $\alpha = 1$; in the right pane, we have used $L = 1.8$. The dashed black lines indicate the theoretical slopes given in the theorems.}
\end{center}
\end{figure}

We remark that Theorems~\ref{thm:phiS_alphaFixed} and~\ref{thm:phiS_LFixed} differ in the key respect that they impose different analyticity requirements on $f$. 
This is due to the behaviour of the parameters in each situation.  Specifically, in the second case, the behaviour of $\alpha$ and $L$ ensure that 
  the asymptotically limiting singularity in the approximation is due to the map $\phiSi(\,\cdot\,;\alpha)$ and not the function $f$.

\subsection{Resolution analysis of $\phiS$}\label{sect:resolutionphiS}
Having analysed the convergence of $\phiS$ we now consider its resolution power.  In particular, we shall show that letting $L \rightarrow \infty$ with $\alpha$ fixed as in Theorem~\ref{thm:phiS_alphaFixed} leads to $\order(\omega^{3/2})$ resolution power, but by varying $\alpha$ and $L$ as described in Theorem~\ref{thm:phiS_LFixed}, we are able to obtain $\order(\omega)$ resolution power.

We begin by considering the function
\begin{align*}
 \FL(y) 
 &= \exp\left( 2 \pi i \omega \phiSi\left( \frac{L}{2} (y-1) \right) \right) \\
 &= \exp\left( 2 i \alpha \omega \log\left(1 + \exp\left( \frac{L\pi}{2\alpha} (y-1) + \frac{\gamma \pi}{\alpha} \right) \right)\right).
\end{align*}
Proceeding as before, we set $y = \cos\theta\cosh\mu  +i\sin\theta\sinh\mu$, and consider the asymptotic behaviour as $\mu \to 0$.  We have
\begin{align*}
 \FL(y) \sim \exp \left( 2 i \alpha \omega K(\theta,\mu,L,\alpha) \right), \quad \mu \rightarrow 0,
\end{align*}
where
\begin{align*}
 K(\theta,\mu,L,\alpha) = \log\left(1 + \exp\left(\frac{L\pi}{2\alpha}(\cos \theta -1 ) + \frac{\gamma\pi}{\alpha} \right) \exp\left(i \frac{L \pi \mu}{2 \alpha} \sin \theta \right)\right).
\end{align*}
By decomposing this expression into polar form, we obtain 
\begin{align*}
 \Im \left( K(\theta,\mu,L,\alpha) \right) 
  & = \tan^{-1} \left( \frac{ \exp\left(\dfrac{L \pi}{2 \alpha}(\cos\theta-1) 
   + \dfrac{\pi \gamma}{\alpha}\right) \sin\left( \dfrac{L \pi \mu}{2 \alpha} \sin \theta \right) }
	  { 1 + \exp\left(\dfrac{L \pi}{2 \alpha}(\cos\theta-1) 
   + \dfrac{\pi \gamma}{\alpha}\right) \cos\left( \dfrac{L \pi \mu}{2 \alpha} \sin \theta \right) } \right) . 
\end{align*} 
Now suppose that $\mu \to 0$ in such a way that $L\mu/\alpha \to 0$.  Then\begin{align*}
 \Im \left( K(\theta,\mu,L,\alpha) \right) 
   \sim \frac{L \pi \mu}{2 \alpha} N(\theta,L,\alpha), 
\end{align*} 
where
\vspace{-1em}
\begin{align*}
  N(\theta,L,\alpha) = \frac{ \exp\left(\dfrac{L \pi}{2 \alpha}(\cos\theta-1) 
   + \dfrac{\pi \gamma}{\alpha}\right) \sin\theta}{1 + \exp\left(\dfrac{L \pi}{2 \alpha}(\cos\theta-1) 
   + \dfrac{\pi \gamma}{\alpha}\right)} .
\end{align*}
As was the case for our analysis of the map $\phiE$, maximising $|\FL(y)|$ over $E_\mu$ corresponds to minimising $N(\theta,L,\alpha)$ over $\theta \in [0,2\pi)$.
Indeed, the relevant stationary points can be determined by examining the quantity
\begin{align*}
 \frac{{\rm d}}{{\rm d} \theta}N(\theta,L,\alpha) = I(\theta,L,\alpha) J(\theta,L,\alpha). 
\end{align*}
Here, $I(\theta,L,\alpha)$ is a function of strictly one sign, and
\begin{align*}
 J(\theta,L,\alpha) = 2 \alpha \cos\theta \left(1 + \exp\left(\dfrac{L \pi}{2 \alpha}(\cos\theta-1) 
   + \dfrac{\pi \gamma}{\alpha} \right) \right) + L\pi(\cos^2\theta - 1).
\end{align*}
Thus, the roots of $J(\theta,L,\alpha)$ correspond precisely to the extrema of $N(\theta,L,\alpha)$.  We now make the following observations.  First, we note that the roots of $J(\theta,L,\alpha)$ are the solutions of $J_{1}(\theta,L,\alpha) = J_{2}(\theta,L,\alpha)$, where
\begin{align*}
J_{1}(\theta,L,\alpha) = 1 + \exp \left ( \frac{L \pi}{2 \alpha} (\cos \theta - 1) + \frac{\pi \gamma}{\alpha} \right ) ,\quad J_{2}(\theta,L,\alpha) = \frac{L \pi}{2 \alpha} \tan \theta \sin \theta.
\end{align*}
Next, we observe that $J(\theta,L,\alpha)$ has precisely one root in the interval $[0,\pi)$. This can be seen as follows. The function $J_{1}(\theta,L,\alpha)$ is monotonically decreasing and positive on $[0,\pi)$, whereas the function $J_{2}(\theta,L,\alpha)$ is monotonically increasing and positive on $[0 , \frac{\pi}{2})$ and monotonically increasing and negative on $(\frac{\pi}{2},\pi]$. Since the function $J_2(\cdot,L,\alpha) : [0,\frac{\pi}{2}) \rightarrow [0,\infty)$ is also surjective, the claim follows. 

Let us denote the root lying in $[0,\pi)$ by $\theta^+$. Since $J(\theta,L,\alpha)$ satisfies $J(\eta,L,\alpha) = -J(2\pi-\eta,L,\alpha)$ for $\eta \in [0,\pi)$, it follows that there is also precisely one root in the interval $(\pi,2\pi]$. We denote this root by $\theta^*$ and note that it satisfies $\theta^* = 2\pi-\theta^+$.

We now split our analysis into the two cases considered in Section \ref{ss:conv_phiS}:

\subsubsection{The case $L \rightarrow \infty$, $\alpha$ fixed}\label{sect:psiSalphaFixedLIncreasing}
We claim that
$$
\theta^{+}_{} \sim \xi^{}_{r} / \sqrt{L},\quad L \rightarrow \infty,
$$
where $\xi^{}_{r}=\xi^{}_{r}(\alpha)$ is the unique positive root of the equation
\begin{equation}
\label{xistar_def}
2 \alpha \left ( 1 + \exp \left ( -\frac{\pi \xi^2}{4 \alpha} + \frac{\pi \gamma}{\alpha} \right ) \right ) - \pi \xi^2 = 0.
\end{equation}
To see this, we merely note that 
$$
J(\xi/\sqrt{L},L,\alpha) \sim 2 \alpha \left ( 1 + \exp \left ( -\frac{\pi \xi^2}{4 \alpha} + \frac{\pi \gamma}{\alpha} \right ) \right ) - \pi \xi^2 + \order(L^{-1}),\quad L \rightarrow \infty.
$$
The first term vanishes only when $\xi = \xi^{}_{r}(\alpha)$.  Therefore, as $L\rightarrow \infty$, the function $N(\theta,L,\alpha)$ is minimised at $\theta = 2 \pi - \xi^{}_{r} / \sqrt{L}$.  Moreover, we have
\begin{align*}
N(2 \pi - \xi^{}_{r} / \sqrt{L},L,\alpha ) 
\sim - \frac{\exp \left ( -\dfrac{\pi \xi^{2}_{r}}{4 \alpha} + \dfrac{\pi \gamma}{\alpha} \right ) \dfrac{\xi^{}_{r}}{\sqrt{L}}  }{1+\exp \left ( -\dfrac{\pi \xi^{2}_{r}}{4 \alpha} + \dfrac{\pi \gamma}{\alpha} \right ) } 
= -\left ( 1 - \frac{2 \alpha}{\pi \xi^{2}_{r}} \right ) \frac{\xi^{}_{r}}{\sqrt{L}},
\end{align*}
where the equality follows from a rearrangement of~\eqref{xistar_def}. (Note also that $\pi \xi^{2}_{r}  > 2 \alpha$.)  Thus, as $\mu \rightarrow 0$, $L \mu / \alpha \rightarrow 0$, $L \rightarrow \infty$  we have shown that~\eqref{eqn:ellipseMax} is 
$$
m(\mu;\FL) \sim \exp \left ( \sqrt{L} \pi \omega \left ( 1 - \frac{2 \alpha}{\pi \xi^{2}_{r}} \right ) \xi^{}_{r} \mu \right ).
$$
Therefore upon substitution into~\eqref{eqn:BernsteinBound}, we obtain the resolution criterion 
$$
n \geq  \sqrt{L} \pi \omega \left ( 1 - \frac{2 \alpha}{\pi \xi^{2}_{r}} \right ) \xi^{}_{r}.
$$
\begin{theorem}\label{res_phiS_Linfty}
Let $\R(\omega;\delta)$ be given by~\eqref{eqn:epsilonResolution} where $\{\Psi^{(n)}\}_{n \in \natural}$ is the approximation scheme defined by~\eqref{eqn:pnLphi} corresponding to the map $\varphi = \phiS$.  Let $\alpha$ be fixed and $L = cn^{2/3}$ for some $c > 0$. Then
if $\xi^{}_{r} = \xi^{}_{r}(\alpha)$ is the unique positive root of~\eqref{xistar_def}, then
\vspace{-1em}
\begin{align}\label{eqn:resolutionFigure_phiS}
 \limsup_{\delta \to 1^{-}} \limsup_{|\omega| \to \infty} \frac{\R(\omega;\delta)}{|\omega|^\frac{3}{2}} \leq \left ( \sqrt{c} \pi \left ( 1 - \frac{2 \alpha}{\pi \xi^{2}_{r}} \right ) \xi^{}_{r} \right )^{3/2}.
\end{align}
\end{theorem}

\subsubsection{The case $\alpha, \alpha / (L-1) \rightarrow 0$}

We claim that $\theta^* \sim 3\pi/2$ in this case, provided $L < 2$. To see this, set $ \theta^* = 3\pi/2 + \delta$, where $\delta \to 0$ as $\alpha \to 0$. The precise asymptotic behaviour of $\delta$ can then be determined as follows. We have  
\begin{align*}
 J\left(\theta^* , L , \alpha \right) 
& \sim 2 \alpha \delta \left( 1 + \exp \left (  \frac{\pi }{\alpha} - \frac{L \pi}{2 \alpha} \right )\right) - L \pi, \qquad \alpha \rightarrow 0.
\end{align*}
Thus, if
\vspace{-1em}
\begin{align*}
\delta \sim \frac{L \pi}{2 \alpha} \left( 1 + \exp \left ( \frac{\pi }{\alpha} - \frac{L \pi}{2 \alpha} \right )\right)^{-1}, \qquad \alpha \to 0,
\end{align*}
then $J(\theta^*,L,\alpha) \sim 0$ as $\alpha \to 0$, provided $L < 2$.  It follows from the antisymmetric nature of $J(\theta,L,\alpha)$ that the other root, $\theta^+$, satisfies $\theta^+ \sim \pi/2$ as $\alpha \to 0$.  Therefore
\begin{align*}
 N(\theta,L,\alpha) \geq N(\theta^*,L,\alpha) \sim -1,\qquad \alpha \to 0.
\end{align*}
Consequently, as $\mu \to 0$, $\alpha \to 0$, $L \mu /\alpha \to 0$, from~\eqref{eqn:ellipseMax} we obtain
$$
m(\mu;\FL) \sim \exp\left( \mu L \omega \pi \right) .
$$
The corresponding convergence criterion upon substitution into~\eqref{eqn:BernsteinBound} is therefore $n \geq L \omega \pi.$  We have thus proven the following result.

\begin{theorem}\label{thm:phiSoptimalResolutionResult}
Let $\R(\omega;\delta)$ be as in Theorem~\ref{res_phiS_Linfty}.  Let $\alpha$ and $L$ be as in Theorem \ref{thm:phiS_LFixed}.  Then the associated resolution constant $r$ satisfies
\begin{align}\label{eqn:phiSoptimalResolutionResult}
r \leq \pi \left ( \limsup_{n \rightarrow \infty} L \right ).
\end{align}
\end{theorem}
We note that the result of Theorem~\ref{thm:phiSoptimalResolutionResult} is not stated simply as $r \leq L \pi$. Though such a formulation would be correct for fixed $L > 1$, the $\limsup$ allows us more generality.

Figure~\ref{fig:numericalExperiments_phiS} provides a numerical experiment verifying Theorem~\ref{thm:phiSoptimalResolutionResult}.

\begin{figure}[ht]
\begin{center}
\begin{overpic}[scale=0.55]{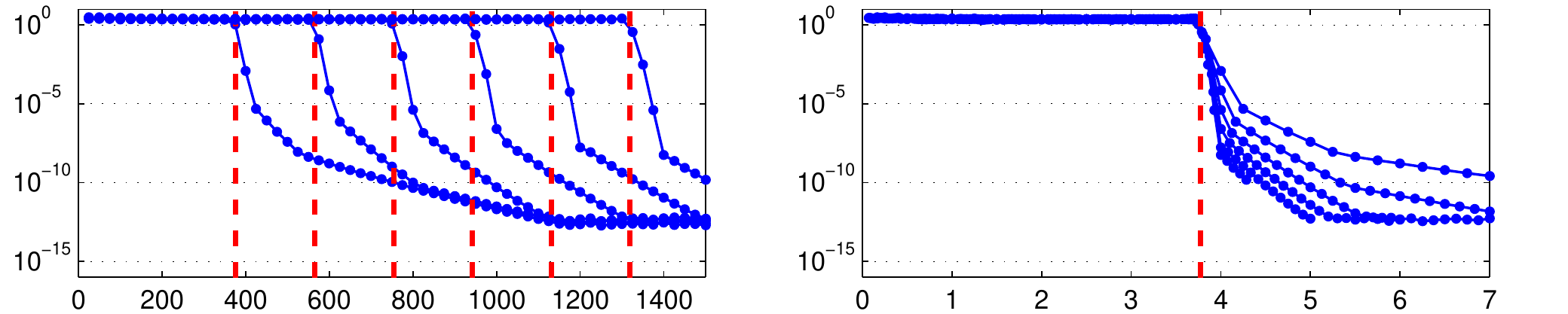}
 \put(25,-3) {\footnotesize $n$}
 \put(73,-3) {\footnotesize $n/\omega$}
\end{overpic}
\vspace{0.5em}
\caption{\label{fig:numericalExperiments_phiS} \small The error in approximating $e^{2\pi i \omega x}$ on $[0,1]$ for $\omega = 50,100,\ldots,350$ using~\eqref{eqn:pnLphi} in conjunction with the map $\varphi = \phiS$. We have set $L = 1.2$ and $\alpha = 0.7/\sqrt{n}$. The dashed red lines indicate the resolution estimate $\approx1.2 \pi \omega \approx 3.7699 \omega$ given by~\eqref{eqn:phiSoptimalResolutionResult}.}
\end{center}
\end{figure}

\section{Infinite-interval maps}\label{sect:inf}

Let $f(x)$ be a function analytic on $(0,1)$ and continuous on $[0,1]$ and suppose that $\psi: (0,1) \mapsto (-\infty,\infty)$ is an invertible mapping.  We again use $\F(s)$ to denote the transplant of $f(x)$ to the $s$-variable:
\begin{align*}
 \F(s) = f(\psi_{}^{-1}(s)), \qquad s \in (-\infty,\infty).
\end{align*}
Proceeding as before, we choose some $L > 0$ and construct an approximation to $\F(s)$ on $[-L,L]$. This can be achieved by defining
\begin{align*}
 \FL(y) = \F(Ly), \qquad y \in (-\infty,\infty),
\end{align*}
which represents a scaling of the interval $[-L,L]$ to $[-1,1]$. 

We will again use the Chebyshev interpolant defined by~\eqref{eqn:chebInterpolant} to approximate the function $\FL(y)$ on $[-1,1]$. Moreover, if $\xL = \psi^{-1}(-L) = 1 - \psi^{-1}(L)$, then the final approximation $\pnL$ to $f(x)$ on $[0,1]$ can be defined as follows:
\begin{align}\label{eqn:pnLpsi}
 \pnL(x) = \left\{\begin{array}{cl}
  \FL(-1),                & \quad    x \in [0,\xL);   \\
  \Pn(\psi(x)/L),         & \quad    x \in [\xL,1-\xL]; \\
  \FL(1),                 & \quad    x \in (1-\xL,1].      
\end{array}\right.
\end{align}
The corresponding error equation is therefore given by
\begin{align}
\nonumber
 \| f-\pnL \| = \max\big\{ \, \|\FL - P_n \|^{}_{y \in [-1,1]}, \qquad \qquad \qquad  \qquad \qquad \qquad \qquad \qquad \\  
\label{eqn:pnLpsiError}    
\qquad  \qquad \qquad \qquad \| \FL-F^{}_L(-1) \|^{}_{y \in (-\infty,-1)} \, , \, 
\| \FL-F^{}_L(1) \|^{}_{y \in (1,\infty)} \, \big\}.
\end{align}

\subsection{Unparameterised exponential map $\psiE$}

The natural analogue of the map $\phiE$ in the infinite interval setting together with its inverse is:
\begin{align*}
 \psiE(x) = \log\left(\frac{x}{1-x}\right), \qquad	
 \psiEi(s) = \frac{\exp(s)}{1+\exp(s)}.
\end{align*}

For our derivation of the new map $\psiSi$, it will be helpful for us to consider the action of the inverse transform $\psiEi$ on an infinite strip. If $\S_\alpha^{}$ is the infinite strip of half-width $\alpha < \pi$, then $\psiEi(\S_\alpha^{})$ is a \emph{lens}-shaped region formed by two circular arcs meeting with half-angle $\alpha$ at $x = 0$ and $x=1$. This is illustrated in Figure~\ref{fig:stripToLens}.

\begin{figure}[ht]
\begin{center}
\begin{overpic}[scale=0.45]{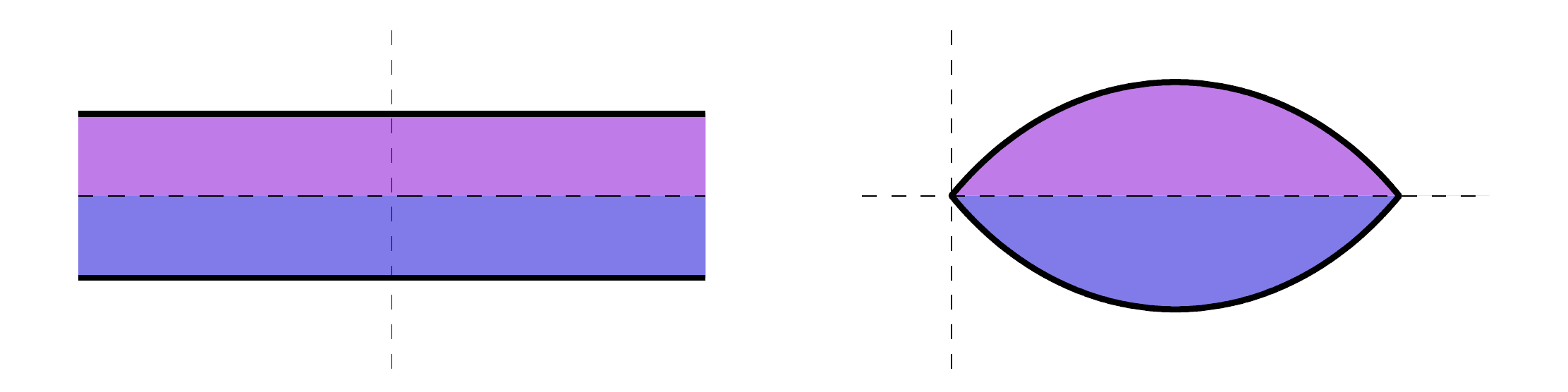}
\put(30,20){\footnotesize $\S^{}_\alpha$}
\put(70,22){\footnotesize $\psiEi(\S^{}_\alpha)$}
\put(58,9){\footnotesize $0$}
\put(90,9){\footnotesize $1$}
\put(1,17.2){\footnotesize $i\alpha$}
\put(-1.5,6){\footnotesize $-i\alpha$}
\end{overpic}
\caption{\label{fig:stripToLens} \small For any $\alpha < \pi$, $\psiEi$ maps the infinite strip of half-width $\alpha$ onto the lens-shaped region formed by two circular arcs meeting with half-angle $\alpha$ at $x=0$ and $x = 1$.}
\end{center}
\end{figure}
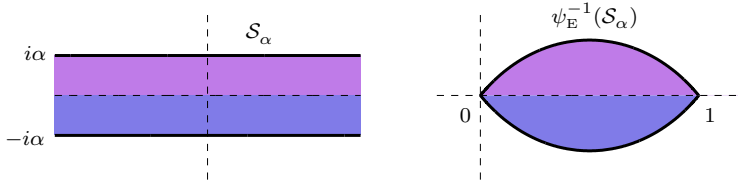

\subsection{Convergence rate for $\psiE$} We quote the relevant result from~\cite{Richardson13}:

\begin{theorem}[\cite{Richardson13}, Thm.~3.2]\label{thm:psiE}
 Let $f(x)$ be analytic and bounded in $\psiEi(\S_\alpha^{})$ for some $0< \alpha < \pi$. Let $f$ satisfy $|f(x)-f(0)| = \order(|x|^\tau)$ as $x \to 0$ and $|f(x)-f(1)| = \order(|1-x|^\tau)$ as $x \to 1$, for some $\tau > 0$. Let $\pnL$ be the approximation defined by~\eqref{eqn:pnLpsi} corresponding to the map $\psi = \psiE$. If $c>0$ and $L = c\sqrt{n}$ then as $n \to \infty$,
 \begin{align*}
  \|f-\pnL\| = \ordera\left( C^{-\sqrt{n}} \right)  , \qquad C = \min\left\{ \exp\left( \alpha / c \right) , \exp \left( \tau c \right) \right\} .
 \end{align*}
 \end{theorem}

\subsection{Resolution analysis of $\psiE$}

As before, we consider the oscillatory complex exponential $f(x) = e^{2\pi i \omega x}$ on $[0,1]$. The transplanted function $\FL(y)$ is
\begin{align*}
 \FL(y) = f(\psiEi(Ly)) = \exp\left( 2 \pi i \omega \left(\frac{\exp(Ly)}{1+\exp(Ly)} \right) \right),
\end{align*}
and again, our task is to obtain an estimate for the quantity $m(\mu;\FL)$ in~\eqref{eqn:ellipseMax}. Letting $y = \cos\theta\cosh\mu + i\sin\theta\sinh\mu$ denote the Bernstein ellipse $E_\mu$, we have
\begin{align*}
 |\FL(y)| = \exp\left( -2 \pi \omega \, G(\theta,\mu,L) \right) ,
\end{align*}
where
\vspace{-1em}
\begin{align*}
 G(\theta,\mu,L) = \Im\left(  \frac{\exp(L(\cos \theta \cosh \mu + i \sin \theta \sinh \mu))}{1+\exp(L(\cos \theta \cosh \mu + i \sin \theta \sinh \mu))} \right).
\end{align*}
It follows that maximising $|\FL(y)|$ over $E_\mu$ is equivalent to minimising $G(\theta,\mu,L)$ over $\theta \in [0,2\pi)$. After some algebra, we find that

\begin{displaymath}
 G(\theta,\mu,L) = \frac{\exp(L\cos\theta\cosh\mu)\sin(L\sin\theta\sinh\mu)}{1+2\exp(L\cos\theta\cosh\mu)\cos(L\sin\theta\sinh\mu)+\exp(2L\cos\theta\cosh\mu)}.
\end{displaymath}
We are interested in particular in the scenario $\mu \to 0$, $L \to \infty$, $\mu L \to 0$, for which 
\begin{align*}
 G(\theta,\mu,L) \sim \frac{1}{2} L \mu P(\theta), \qquad 
 P(\theta) = \frac{\sin\theta}{1+\cosh(L \cos \theta)}.
\end{align*}
For any $L > 0$, one can verify that $P(3\pi/2) = -1/2$ is the global minimum of $P(\theta)$ over $[0,2\pi)$. Thus in~\eqref{eqn:ellipseMax}, as $\mu \to 0$, $L \to \infty$, $\mu L \to 0$, we have 
\begin{align*}
 m(\mu;\FL) \sim \exp\left( \frac{\pi}{2} L \omega \mu \right).
\end{align*}
Upon substituting this into~\eqref{eqn:BernsteinBound}, we then see that the resolution criterion is
\begin{align}\label{eqn:resolutionFigure_psiEwithoutc}
 n \geq \frac{\pi}{2} L \omega.
\end{align}

\begin{theorem}
Let $\R(\omega;\delta)$ be given by~\eqref{eqn:epsilonResolution} where $\{\Psi^{(n)}\}_{n \in \natural}$ is the approximation scheme defined by~\eqref{eqn:pnLpsi} corresponding to the map $\psi = \psiE$. Let $L = c\sqrt{n}$. Then 
\vspace{-1em}
\begin{align}\label{eqn:resolutionFigure_psiE}
 \limsup_{\delta \to 1^{-}} \limsup_{|\omega| \to \infty} \frac{\R(\omega;\delta)}{|\omega|^{2}} \leq  \left(\frac{\pi c}{2}\right)^{2}.
\end{align}
\end{theorem}

In summary, the numerical method based upon the infinite interval transform $\psiE$ converges at the rate $\ordera(C^{-\sqrt{n}})$, but requires $\order(\omega)$ points-per-wavelength, and therefore $\order(\omega^2)$ points in total, to resolve the oscillatory complex exponential $e^{2\pi i \omega x}$.

A numerical verification of this result is provided in Figure~\ref{fig:numericalExperiments_psiE}.

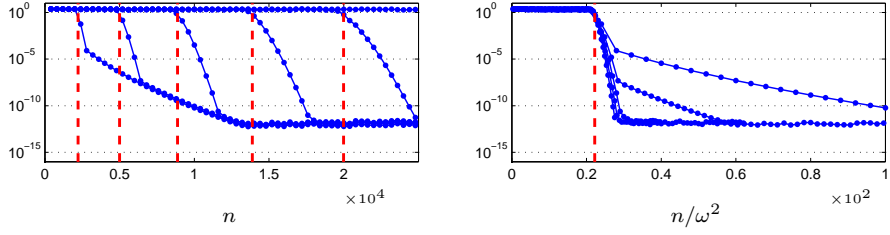
\begin{figure}[ht]
\begin{center}
\begin{overpic}[scale=0.55]{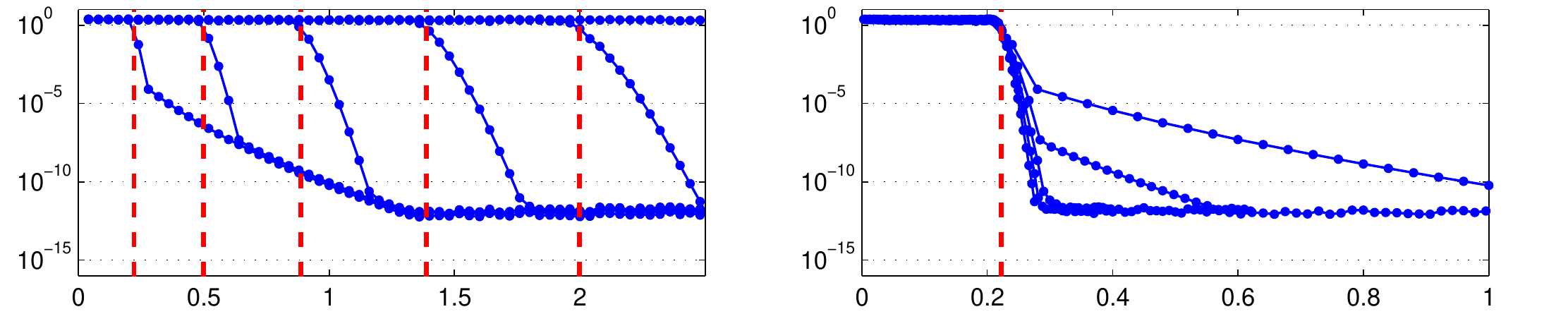}
 \put(24,-4) {\footnotesize $n$}
 \put(37,-2) {\tiny $\times 10^4$}
 \put(72,-4) {\footnotesize $n/\omega^{2}$}
 \put(89,-2) {\tiny $\times 10^2$}
\end{overpic}
\vspace{1em}
\caption{\label{fig:numericalExperiments_psiE} \small The error $\|f - \pnL\|$ in approximating $f(x) = e^{2 \pi i \omega x}$ on $[0,1]$ using~\eqref{eqn:pnLpsi} in conjunction with the map $\psi = \psiE$ for $\omega = 100, 150, \ldots, 350$. In each case, we use the value of $c = 0.3$ in the relationship $L = c \sqrt{n}$.  The dashed red lines indicate the theoretical resolution figure ($\approx  0.222 \omega^2$) given by~\eqref{eqn:resolutionFigure_psiE} corresponding to this particular choice of $c$. }
\end{center}
\end{figure}

\subsection{Parameterised exponential map $\psiS$} We now introduce an analogue of the slit-strip map $\phiS$ for the infinite interval setting. As before, this will lead us to a numerical method requiring a fixed number of points-per-wavelength to resolve the complex exponential $e^{2 \pi i \omega x}$ as $\omega \to \infty$.

The derivation proceeds much as before.  We begin with the infinite strip $\S_\alpha^{}$ in the $s$-variable, and construct a map via the Schwarz--Christoffel formula which takes us to the \emph{two-slit strip} $\S_\alpha^{} \backslash \{(-\infty,0] \cup [1,\infty)\}$ in the $x$-variable. See Figure~\ref{fig:stripToTwoSlitStrip}. 

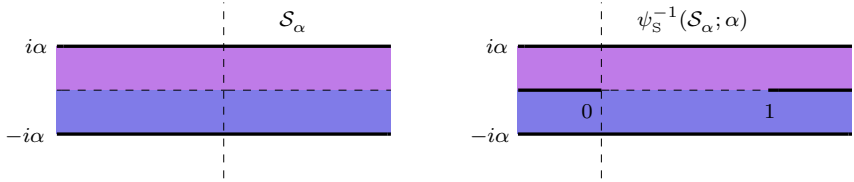
\begin{figure}[ht]
\begin{center}
\begin{overpic}[scale=0.5]{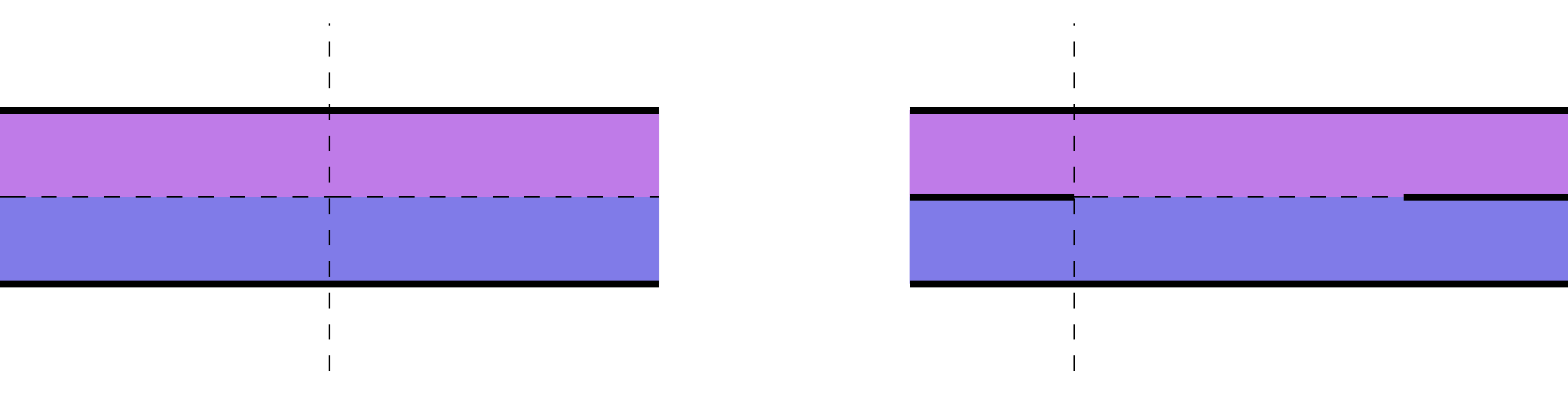}
 \put(28,20.5) {\footnotesize $\S^{}_\alpha$}
\put(-4,17.2){\footnotesize $i\alpha$}
\put(-6.5,6){\footnotesize $-i\alpha$}
 \put(73,20.5) {\footnotesize $\psiSi(\S^{}_\alpha;\alpha)$}
\put(54,17.2){\footnotesize $i\alpha$}
\put(52,6){\footnotesize $-i\alpha$}
\put(66,9){\footnotesize $0$}
\put(89,9){\footnotesize $1$}
\end{overpic}
\vspace{-1em}
\caption{\label{fig:stripToTwoSlitStrip} \small The transformation $\psiSi(\,\cdot\,;\alpha)$ maps the infinite strip of half-width $\alpha$ 
to the infinite two-slit strip of half-width $\alpha$.}
\end{center}
\end{figure}

For brevity's sake we forgo the details of the construction, and merely note that the corresponding two-slit map and its inverse are given by 
\begin{align*}
 \psiSi(s;\alpha) &= 
\frac{\alpha}{\pi} \log \left( \frac{  1 + \exp(\pi(s+1/2)/\alpha) }
{ 1 + \exp(\pi(s-1/2)/\alpha) } \right), \\
\psiS(x;\alpha) &= \frac{\alpha}{\pi} \log\left( \frac{\exp(\pi x / \alpha) - 1}
{1 - \exp(\pi(x-1)/\alpha)} \right) -\frac{1}{2} .
\end{align*}

\subsection{Convergence analysis of $\psiS$}\label{ss:conv_psiS} 

As with the analysis of $\phiS$, we have two situations to consider. The first involves fixing $\alpha$ and allowing $L \to \infty$ as $n \to \infty$; the second involves keeping $L$ bounded and allowing $\alpha \to \infty$ as $n\ \to \infty$.

\begin{theorem}\label{thm:psiS_alphaFixed} 
Let $f(x)$ be analytic and bounded in the region $\psiSi(\S^{}_{\beta};\alpha)$ for some $0 < \beta \leq \alpha$. Let $f$ satisfy $|f(x)-f(0)| = \order(|x|^\tau)$ as $x \to 0$ and $|f(x)-f(1)| = \order(|1-x|^\tau)$ as $x \to 1$ for some $\tau > 0$. Let $\pnL$ be the approximation defined by~\eqref{eqn:pnLpsi} corresponding to the map $\psi = \psiS(\,\cdot\,;\alpha)$. Then for any $c > 0$ the choice $L = c\sqrt{n}$ minimises the error as $n \to \infty$, in which case we have
\begin{align}\label{thm:psiS_alphaFixedResult2} 
    \|f - \pnL\| = \ordera \s \left(C^{-\sqrt{n}}\right) ,\qquad C = \min \left \{ \exp\left ( \beta / c  \right ),  \exp\left( \pi \tau c / \alpha \right ) \right \}.
   \end{align}
\end{theorem}
\begin{proof} The proof is similar to that of the corresponding result for $\phiS$.  Hence we omit the details.
\end{proof}

\begin{theorem}\label{thm:psiS_LFixed}
Let $f$ be analytic in the two-slit strip $\S^{}_{\beta} \backslash \left\{ (-\infty,0] \cup [1,\infty) \right\}$ for some $\beta > 0$. Let $f$ satisfy $|f(x)-f(0)| = \order(|x|^\tau)$ as $x \to 0$ and $|f(x)-f(1)| = \order(|1-x|^\tau)$ as $x \to 1$, for some $\tau > 0$. Let $\pnL$ be the approximation defined by~\eqref{eqn:pnLpsi} corresponding to the map $\psi = \psiS(\,\cdot\,;\alpha)$. If $L > 1/2$ is bounded for all $n$, and $\alpha , \alpha/(L-1/2) \rightarrow 0$ as $n \rightarrow \infty$, then
{\small 
\begin{align}\label{thm:psiS_LFixedResult1}
\| f - p^{(n)}_L \| = \order \left ( \max \left \{  \frac{\sqrt{L^2-1/4}}{\alpha} \exp\left ( - \frac{\alpha n}{\sqrt{L^2-1/4}} \right) , \exp\left( -\frac{\pi \tau}{\alpha} (L-1/2) \right )  \right \} \right ),
\end{align}
}
as $n \rightarrow \infty$.  In particular, the error is minimised when $\alpha = \alpha_0 / \sqrt{n}$ and $L = 1/2 + L_0$ for $\alpha_0,L_0>0$, in which case we have
{\small
\begin{align}\label{thm:psiS_LFixedResult2}
\| f - p^{(n)}_L \| = \ordera \s \left ( C^{-\sqrt{n}} \right ), \quad C = \min \left \{ \exp\left( \frac{\alpha_0}{\sqrt{L_0(1+L_0)}} \right) , \exp\left( \frac{\pi \tau L_0}{\alpha_0} \right) \right \}.
\end{align}
}
\end{theorem}

\begin{proof}
As in Theorem~\ref{thm:phiS_LFixed}, for all sufficiently small $\alpha$, the maximal Bernstein ellipse parameter $\mu$ corresponding to the function $\FL$ is determined by the singularities of $\psiSi$ at $s = \pm \frac{1}{2} \pm i \alpha$. The precise asymptotic behaviour of $\mu$ can be determined as follows. We consider 
\begin{align*}
 \cos \theta \cosh \mu + i \sin \theta \sinh \mu = 1 / (2L) + i \alpha / L.
\end{align*}
Upon eliminating $\theta$ from this equation, one finds that
\begin{align*}
 1/(4 L^2 \cosh^2_{}\mu) + \alpha^2 / (L^2 \sinh^2 \mu) = 1 ,
\end{align*}
from which it is possible deduce that $\mu$ satisfies
\begin{align*}
 \mu \sim \alpha / \sqrt{L^2-1/4} , \qquad \alpha \to 0.
\end{align*}
Substituting this into~\eqref{eqn:BernsteinBound} then gives us the first error term in~\eqref{thm:psiS_LFixedResult1}. 

Next we consider the domain contributions. As $\alpha, \frac{\alpha}{L - 1/2} \to 0$, we find that
\begin{align*} 
\xL = 1 - \psiSi(L) = \psiSi(-L) \sim (\alpha/\pi) \exp\left( - \pi (L - 1/2) / \alpha \right).
\end{align*}
Combining this with the assumptions on $f$ then gives the remaining term in~\eqref{thm:psiS_LFixedResult1}.

To obtain the second part of the result~\eqref{thm:psiS_LFixedResult2}, we observe that both terms in~\eqref{thm:psiS_LFixedResult1} decay at the same rate if, for some $\nu > 0$, we have
\begin{align*}
 \frac{L - 1/2}{\alpha} \sim \nu \frac{\alpha n}{\sqrt{L^2-1/4}} \quad \Leftrightarrow \quad 
 (L - 1/2)^2_{}(L^2 - 1/4) \sim \nu^2 \alpha_{}^4 n_{}^2.
\end{align*}
Since $L$ is assumed to be bounded independently of $n$, this implies that $\alpha = \order(1/\sqrt{n})$ as $n \to \infty$. In other words, we must require that $\alpha$ decays at least as fast as $1/\sqrt{n}$.
On the other hand, some further manipulation gives
\begin{equation}\label{winter}
(L-1/2)^{3/2}_{} = \frac{\nu \alpha^2 n}{(L+1/2)^{1/2}_{}}.
\end{equation}
Hence, the rate of exponential convergence is
$$
\frac{L-1/2}{\alpha} = \frac{\nu^{2/3} \alpha^{1/3} n^{2/3}}{(L+1/2)_{}^{1/3}}.
$$
The right hand side of this equation suggests that we require the slowest possible decay of $\alpha$. This corresponds to choosing $\alpha = \alpha^{}_{0} / \sqrt{n}$ for some $\alpha^{}_{0} > 0$, giving
$$
\frac{L-1/2}{\alpha}  = \frac{\nu^{2/3} \alpha^{1/3}_0 n^{1/2}}{(L+1/2)_{}^{1/3}}.
$$
Using~\eqref{winter}, we then see that $(L-1/2)_{}^{3/2} (L+1/2)^{1/2}_{} = \nu \alpha^2_0$.  Thus, for this choice of $\alpha$, we see that $L$ is necessarily independent of $n$ and we can write $L = 1/2 + L^{}_0$ for some $L^{}_0 > 0$. This completes the proof.
\end{proof}

Figure~\ref{fig:convOneSided} provides numerical experiments verifying Theorems~\ref{thm:psiS_alphaFixed} and~\ref{thm:psiS_LFixed}.

\begin{figure}[ht]
\begin{center}
\begin{overpic}[scale=0.7]{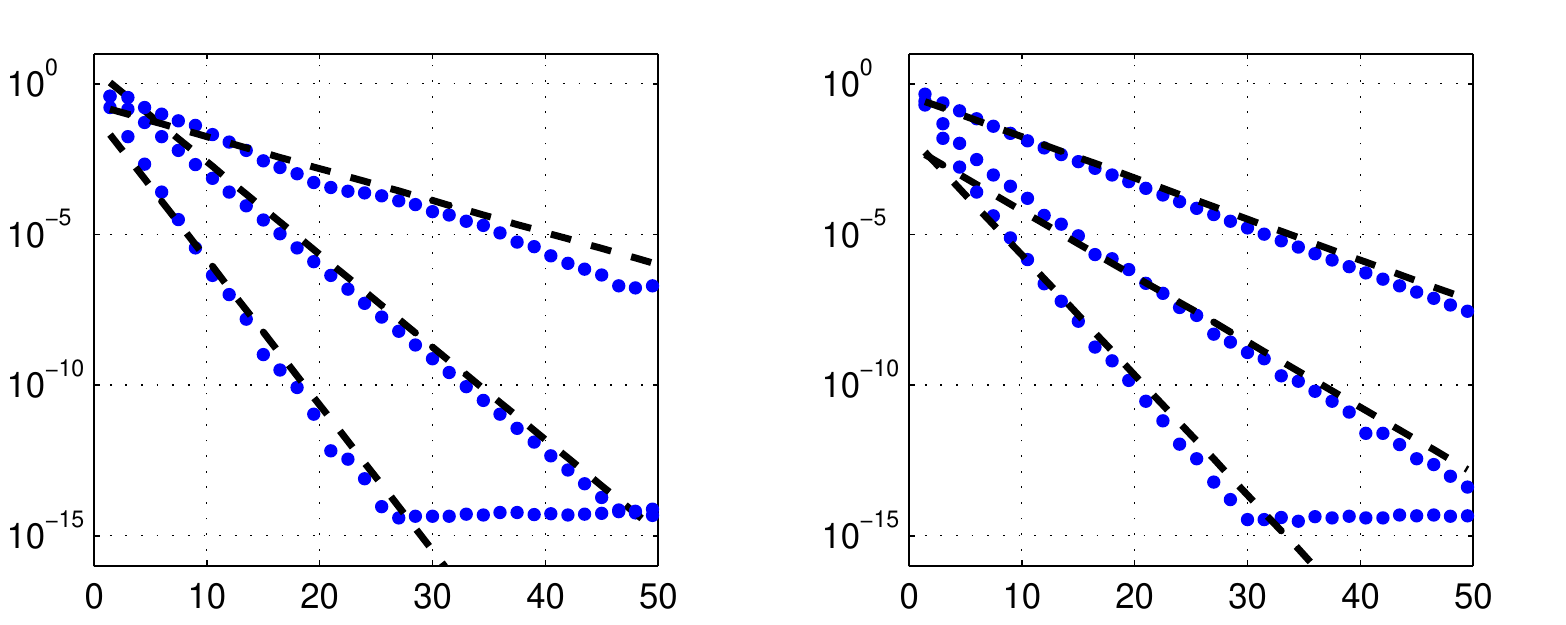}
 \put(13,39){\small $\alpha$ fixed, $L = c\sqrt{n}$}
 \put(64,39){\small $L$ fixed, $\alpha = \alpha^{}_0/\sqrt{n}$}
 \put(21,-4){\small $\sqrt{n}$}
 \put(31,16){\small $c = 0.45$}
 \put(13.5,7){\small $c = 0.9$}
 \put(32,27){\small $c = 4.1$}
 \put(74,-4){\small $\sqrt{n}$}
 \put(89,15){\small $\alpha^{}_{0} = 0.6$}
 \put(65,7){\small $\alpha^{}_{0} = 1.1$}
 \put(83,26){\small $\alpha^{}_{0} = 4$}
\end{overpic}
\vspace{1em}
\caption{\label{fig:convTwoSided} \small Numerical verification of Theorem~\ref{thm:psiS_alphaFixed} (left) and Theorem~\ref{thm:psiS_LFixed} (right) for approximation of the function $f(x) = \sqrt{x}$ on $[0,1]$ using numerical methods based upon the map $\psiS(\,\cdot\,;\alpha)$. In the left pane, we have used the value $\alpha = 1$; in the right pane, we have used $L = 1.3$. The dashed black lines indicate the theoretical slopes given in the theorems.}
\end{center}
\end{figure}

\subsection{Resolution analysis of $\psiS$} 
In order to assess the resolution properties of $\psiS$, we consider the transplanted function
\begin{align*}
 \FL(y) = f(\psiSi(Ly)) = \exp\left( 2 \alpha i \omega \log \left( \frac{  1 + \exp(\pi(Ly+1/2)/\alpha) }{ 1 + \exp(\pi(Ly-1/2)/\alpha) } \right) \right).
\end{align*}
Letting $y = \cos \theta \cosh \mu + i \sin \theta \sinh \mu$, we have
\begin{align*}
 |\FL(y)| = \exp\left( -2\alpha \omega R(\theta,\mu,\alpha,L) \right),
\end{align*}
where $R(\theta,\mu,L) $ is given by
\begin{align*}
 R(\theta,\mu,L) = \Im \left( \log \left( \frac{  1 + \exp(\pi(L(\cos \theta \cosh \mu + i \sin \theta \sinh \mu)+1/2)/\alpha) }{ 1 + \exp(\pi(L(\cos \theta \cosh \mu + i \sin \theta \sinh \mu)-1/2)/\alpha) } \right) \right).
\end{align*}
After some manipulation it is possible to rewrite this as
\begin{align*}
 R(\theta,\mu,\alpha,L) = \Im\left( \log \left( \frac{T^{}_{+}\exp(i\tau^{}_{+})}{T^{}_{-} \exp(i\tau^{}_{-}) } \right) \right) 
= \tau^{}_{+} - \tau^{}_{-},
\end{align*}
where 
\begin{align*}
 \tau^{}_{\pm} = \tan^{-1} \left(  \frac{ \exp\left( \pi \left( L \cos \theta \cosh \mu \pm 1/2 \right) / \alpha \right) \sin( \pi L (\sin \theta \sinh \mu) / \alpha) }{1 + \exp\left( \pi \left( L \cos \theta \cosh \mu \pm 1/2 \right) / \alpha \right) \cos( \pi L (\sin \theta \sinh \mu) / \alpha) } \right) .
\end{align*}
If $L$ and $\alpha$ are such that $L \mu/\alpha \to 0$ as $\mu \to 0$, then 
\begin{align*}
   \tau^{}_{\pm} 
&\sim  \tan^{-1} \left(  \frac{ \exp\left( \pi \left( L \cos \theta \pm 1/2 \right) / \alpha \right) (\pi L \mu/ \alpha) \sin \theta }{1 + \exp\left( \pi \left( L \cos \theta \pm 1/2 \right) / \alpha \right) } \right),  \\
&\sim \frac{\pi L \mu}{\alpha} \sin \theta \frac{ \exp\left( \pm \pi/(2\alpha) \right) }{\exp(-(\pi L/\alpha) \cos \theta) + \exp\left( \pm \pi/(2\alpha) \right) } .
\end{align*}
We therefore have 
\begin{displaymath}
 R(\theta,\mu,\alpha,L) \sim  \frac{\pi L \mu}{\alpha} \sin \theta \, p(\theta,\alpha,L),
\end{displaymath}
where
\begin{displaymath}
 p(\theta,a,L) = \frac{\sinh ( \pi / (2 \alpha) )}{\cosh ( (\pi L / \alpha) \cos \theta ) + \cosh (\pi / (2 \alpha) ) }.
\end{displaymath}
Observe that $p(\theta,\alpha,L)$ is (i) positive for all $\theta$, and (ii) maximised in $\theta$ when $\cos \theta = 0$. Since $\sin \theta$ is maximised at $\theta = \pi/2$ and minimised at $\theta = 3\pi/2$ (at which point it takes a negative value), it follows that the product $\sin \theta p(\theta,\alpha,L)$ is minimised over $[0,2\pi)$ at $\theta = 3\pi/2$. Thus, as $\mu, L \mu/\alpha \to 0$, we have 
\begin{displaymath}
 R(3\pi/2,\mu,\alpha,L) \sim  -  \frac{\pi L \mu}{\alpha} B(\alpha), 
\end{displaymath}
where $B(\alpha) = \frac{\sinh(\pi/(2\alpha))}{1+\cosh(\pi/(2\alpha))}$.  Therefore in~\eqref{eqn:ellipseMax} we find that
\begin{align}\label{eqn:m(mu)psiS}
 m(\mu;\FL) \sim \exp\left( 2 \pi L \omega \mu B(\alpha) \right).
\end{align}
We now divide our analysis into the two parameter regimes described above.

\subsubsection{The case $L \to \infty$, $\alpha$ fixed} Using~\eqref{eqn:m(mu)psiS}, we have the resolution criterion $n \geq 2 \pi L \omega B(\alpha)$.  This immediately gives us the following result:
\begin{theorem}\label{thm:res_psiS_Linfty}
 Let $\R(\omega;\delta)$ be given by~\eqref{eqn:epsilonResolution} where $\{\Psi^{(n)}\}_{n \in \natural}$ is the approximation scheme defined by~\eqref{eqn:pnLpsi} corresponding to the map $\psi = \psiS$.  Let $\alpha$ be fixed and let $L = c\sqrt{n}$ for some $c > 0$. Then, 
\begin{align*}
 \limsup_{\delta \to 1^{-}} \limsup_{|\omega| \to \infty} \frac{\R(\omega;\delta)}{|\omega|_{}^{2}} \leq \left ( 2 \pi c B(\alpha) \right )^2_{}.
\end{align*}
\end{theorem}

\subsubsection{The case $\alpha , \alpha/(L-1/2) \to 0$} Observe that $B(\alpha) \to 1$ as $\alpha \to 0$. Thus, in this case we have the resolution criterion $n \geq 2 \pi L \omega$, and therefore

\begin{theorem}\label{thm:res_phiS_alpha}
 Let $\R(\omega;\delta)$ be as in Theorem~\ref{thm:res_psiS_Linfty}. Let $\alpha$ and $L$ be as in Theorem~\ref{thm:psiS_LFixed}. Then the associated resolution constant $r$ is as follows: 
\begin{align}\label{eqn:resolutionFigure_psiS}
 r \leq 2 \pi \left( \limsup_{n \to \infty} L \right).
\end{align}
\end{theorem}
Figure~\ref{fig:numericalExperiments_psiS} provides a numerical verfication of Theorem~\ref{thm:res_phiS_alpha}.

\begin{figure}[ht]
\begin{center}
\begin{overpic}[scale=0.55]{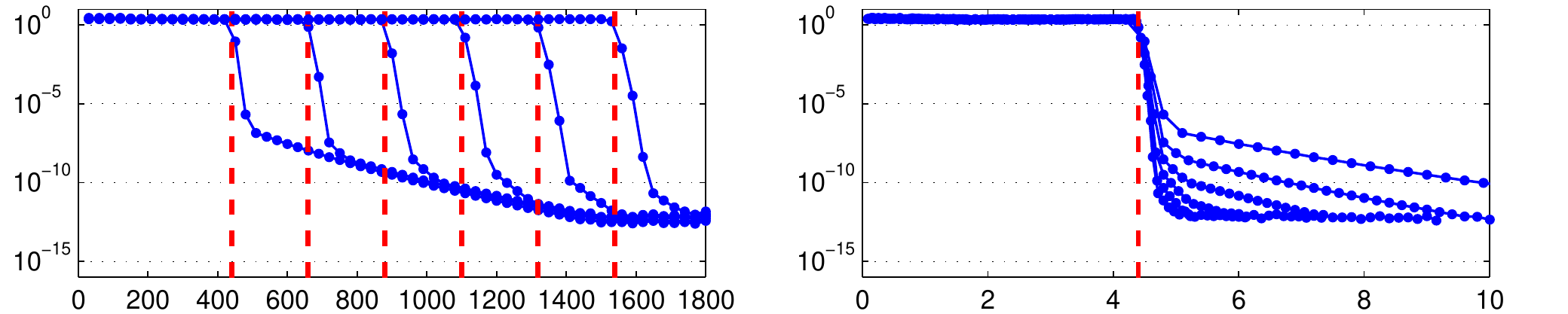}
 \put(25,-3) {\footnotesize $n$}
 \put(73,-3) {\footnotesize $n/\omega$}
\end{overpic}
\vspace{0.5em}
\caption{\label{fig:numericalExperiments_psiS} \small The error in approximating $e^{2\pi i \omega x}$ on $[0,1]$ for $\omega = 50,100,\ldots,350$ using~\eqref{eqn:pnLpsi} in conjunction with the map $\psi = \psiS(\,\cdot\,;\alpha)$. For these experiments, we have set $L = 0.7$ and $\alpha = 0.8/\sqrt{n}$. The dashed red lines indicate the estimate $1.4 \pi \omega \approx 4.3982 \omega$ given by~\eqref{eqn:resolutionFigure_psiS}.
}
\end{center}
\end{figure}

\section{Parameter choices}\label{parameterChoices}
Thus far, we have seen that appropriate parameter choices lead to a finite resolution constant $r > \pi$.  In this final section of the paper, we consider an alternative approach to choosing the parameters $\alpha$ and $L$ which leads to a numerical scheme with formally optimal resolution properties, requiring $\pi$ ppw to begin resolving $e^{2\pi i \omega x}$ as $\omega \to \infty$. To do this, we give up classical convergence of the approximation -- that is, convergence down to $0$ as $n \rightarrow \infty$ -- and instead seek only a finite, but user-controllable error tolerance $0 < \varepsilon < 1$.  Although in theory the difference between these two scenarios is significant, in numerical computations we are always limited by machine epsilon.  Hence if the maximal accuracy $\varepsilon$ is pre-selected to be sufficiently small then we lose little over insisting on classical convergence for all $n$.  Moreover, as we shall show, we get a significant improvement in the effective convergence rate in the regime before such a tolerance is reached.

Approaches related to this technique have been used before in several other contexts. Most notably, Kosloff and Tal-Ezer~\cite{KTE} used a similar idea to improve timestepping restrictions in spectral methods for PDEs. See also~\cite{AdcockHuybrechs}.

\subsection{Parameter choices for \boldmath{$\phiS$}}\label{parameterChoicesphiS}
This approach is based on the error expression~\eqref{thm:phiS_LFixedResult1}, and the idea is as follows.  We first equate one of the terms in~\eqref{thm:phiS_LFixedResult1} to $\varepsilon$, and second, we select $\alpha$ and $L$ such that the other term decays as quickly as possible subject to the existing constraints in Theorem \ref{thm:phiS_LFixed}.

Consider the first step.  Since $\tau$ will in general be unknown, we elect to use the first term in~\eqref{thm:phiS_LFixedResult1}.  Thus, we pick $\alpha$ and $L$ such that
\begin{align*}
\frac{\alpha n}{\sqrt{L-1}}  = | \log \varepsilon | \quad 
\Leftrightarrow \quad
L = 1 + \frac{n^2 \alpha^2}{|\log \varepsilon|^2}.
\end{align*}
Substituting this expression for $L$ into the second term gives
\begin{align*}
\frac{L-1}{\alpha} = \frac{\alpha n^2}{|\log \varepsilon|^2}. 
\end{align*}
Since $L$ is assumed to be bounded, this equation implies that $\alpha = \order(n^{-1})$.  On the other hand, since $\alpha / (L-1) \rightarrow 0$ for Theorem \ref{thm:phiS_LFixed} to hold, this equation also states that  $1/ \alpha = o(n^2)$. Suppose that $\alpha$ and $L$ are as just described, then~\eqref{thm:phiS_LFixedResult1} implies
\begin{align*}
\| f - \pnL \| = \ordera \left ( \max \left \{ \, \varepsilon \, , C^{-\alpha n^2} \right \} \right ),\qquad C = \exp \left ( \pi \tau/ | \log \varepsilon |^2 \right ),
\end{align*}
as $n \to \infty$. In particular, if we set
\begin{align}\label{eqn:phiSparams}
\alpha = \sigma | \log \epsilon | n^{p-2},\quad L = 1 + \sigma^2 n^{2p-2},
\end{align}
for some $\sigma > 0$ and $0 < p < 1$, then as $n \to \infty$,
\begin{align}\label{eqn:err_asymp}
\| f - \pnL \| = \ordera \left ( \max \left \{ \, \varepsilon \, , C^{-n^{p}} \right \} \right ),\qquad C = \exp \left ( \pi \tau \sigma / | \log \varepsilon | \right ).
\end{align}
Hence, asymptotically the error behaves like $C^{-n^p}$ until $n$ is sufficiently large such that this term is less than $\epsilon$.  In other words, we obtain exponential convergence at rate $p$ down to a maximal achievable accuracy proportional to $\epsilon$.  Moreover, since the associated choice of parameters has $L \to 1$ as $n \to \infty$, by Theorem~\ref{thm:phiSoptimalResolutionResult} we also obtain $\pi$ ppw asymptotically.  Thus, by conceding convergence for all $n$, we obtain formally optimal resolution power, in the sense that the ppw figure is determined by that of the underlying approximation (i.e.\ Chebyshev interpolants).

Interestingly, by making such a parameter choice we can actually obtain a \emph{better} rate of convergence than that of the exponential map $\phiE$ (recall that the convergence rate therein is exponential with index $2/3$).  We stress at this point, however, that the estimate~\eqref{eqn:err_asymp} holds asymptotically as $n \to \infty$, and not for fixed $n$.  Thus it is not an error bound per se, and in practice one may see slower convergence for small $n$ before the onset of the predicted asymptotic behaviour. 

Figure~\ref{fig:optimalParamsphiS} provides a numerical experiment corroborating the above arguments.

\begin{figure}[ht]
\begin{center}
\begin{overpic}[scale=0.55]{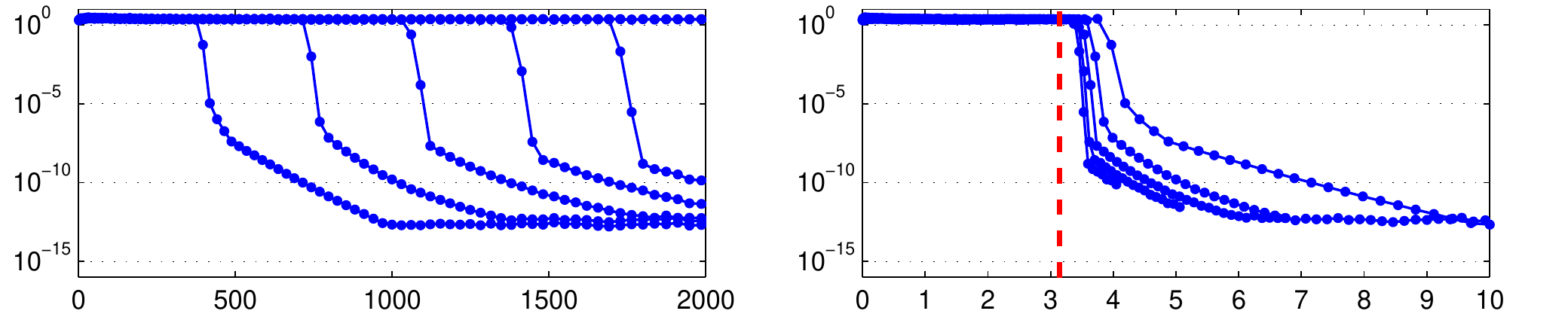}
 \put(25,-3) {\footnotesize $n$}
 \put(73,-3) {\footnotesize $n/\omega$}
\end{overpic}
\vspace{0.5em}
\caption{\label{fig:optimalParamsphiS} \small The error in approximating $e^{2\pi i \omega x}$ on $[0,1]$ for $\omega = 50,100,\ldots,350$ using~\eqref{eqn:pnLphi} in conjunction with the map $\varphi = \phiS(\,\cdot\,;\alpha)$, where $\alpha$ and $L$ are set according to the relationships given by~\eqref{eqn:phiSparams}. We have used the values $\sigma = 3.5$, $p = 2/3$ and $\varepsilon = 2^{-52}$.}
\end{center}
\end{figure}

\subsection{Parameter choices for \boldmath{$\psiS$}}

A very similar analysis can be performed for $\psiS$.  Arguing in the same way, one obtains the following.  If, for some arbitrary $\sigma > 0$, $p \in (0,1]$ we set
\begin{align}\label{eqn:psiSparams}
 \alpha = \sigma |\log \varepsilon| n_{}^{p-2}, \quad L = \sqrt{1/4 + \sigma^2n_{}^{2p-2}},
\end{align}
 then, as $n \to \infty$ we have
\begin{align*}
 \| f - \pnL \| = \ordera \left( \max\left\{ C^{-n_{}^{p}} , \varepsilon \right\} \right), 
 \end{align*}
 where
\begin{align*}
 C = \left \{ \begin{array}{lc} \exp(\pi \tau \sigma/|\log \varepsilon|)& 0<p<1 \\ \exp \left ( \pi \tau \left ( \sqrt{1/4+\sigma^2}-1/2 \right ) / (\sigma | \log \epsilon |)\right ) & p=1. \end{array} \right .
\end{align*}
Much as in the previous case, we draw the same conclusions.  By forfeiting classical convergence for all $n$, we are able to obtain both optimal resolution and a faster asymptotic decay of the first error term.  See Figure~\ref{fig:optimalParamspsiS} for a numerical experiment.

\begin{figure}[ht]
\begin{center}
\begin{overpic}[scale=0.55]{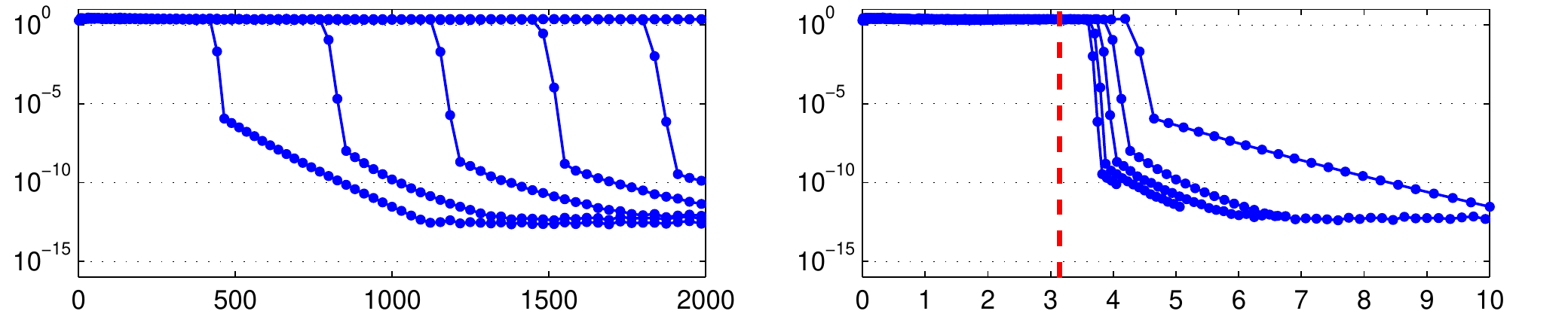}
 \put(25,-3) {\footnotesize $n$}
 \put(73,-3) {\footnotesize $n/\omega$}
\end{overpic}
\vspace{0.5em}
\caption{\label{fig:optimalParamspsiS} 
 The error in approximating $e^{2\pi i \omega x}$ on $[0,1]$ for $\omega = 50,100,\ldots,350$ using~\eqref{eqn:pnLpsi} in conjunction with the map $\psi = \psiS(\,\cdot\,;\alpha)$, where $\alpha$ and $L$ are set according to the relationships given by~\eqref{eqn:psiSparams}. Here we have used the values $\sigma = 3.5$, $p = 2/3$ and $\varepsilon = 2^{-52}$.}
\end{center}
\end{figure}

\section{Practical considerations}

There are a few issues of which one should be aware when implementing numerical methods based on the slit-strip maps $\phiS$ and $\psiS$ in finite precision arithmetic. First, if implemented naively, the formulas for $\phiS,\phiSi$ and $\psiS,\psiSi$ may, in certain circumstances, result in cancellation errors. Fortunately, these effects can largely be avoided by approximating terms such as $\exp(x)-1$ and $\log(1+x)$ using an appropriate series expansion near $x=0$. (This is taken care of automatically in Matlab by the \texttt{expm1} and \texttt{log1p} functions.)

A further important issue concerns the fact that there is practical limit on the smallest permissible strip-width parameter $\alpha$. In double precision, the limit is approximately $0.005$ for both $\phiS$ and $\psiS$. One can understand why this is the case by observing that the quantity $\exp(\pi/0.004)$ is larger than the largest double-precision floating point number. There may be a way to circumvent this issue, but such matters are beyond the scope of the present investigation.

This second point has implications for the parameter choices of Section~\ref{parameterChoices}.
In particular, it may sometimes be awkward to choose parameters which give both good convergence {\em and} good resolution. Take Section~\ref{parameterChoicesphiS} for example, where one chooses $p$ and $\sigma$ and then sets $\alpha = \sigma|\log\varepsilon|n_{}^{p-2}$, $L = 1 + \sigma^2n^{2p-2}$. One can see from these formulas that $\alpha$ decreases at least as fast as $1/n$. Thus, in order to ensure that the $0.005$ limit is not breached too quickly, one should choose a larger $\sigma$. But this has the effect of increasing $L$ and hence worsening the resolution for finite $\omega$. Fortunately, however, as we demonstrate in Figures~\ref{fig:optimalParamsphiS} and~\ref{fig:optimalParamspsiS}, it is in practice usually possible to find a choice of parameters which balance these effects. 

\section{Concluding remarks}\label{sect:conc}

The focus of this paper is the introduction of new exponential variable transform methods to approximate functions with singularities at one or both endpoints.  Once the function $f(x)$ has been mapped accordingly, one next applies an appropriate approximation strategy on either the infinite or semi-infinite interval.  Throughout we have used domain truncation, followed by Chebyshev interpolation, for this purpose.  This choice was made for its simplicity and ubiquity, and because it allows one to make a mathematical comparison that confirms the advantage of the new mappings.

However, this approximation is by no means the only option, and there are definite advantages to other choices.  For example, when using the infinite interval maps $\psiE$ and $\psiS$, Fourier series would certainly be viable options.  Quoting Boyd, we note that ``Chebyshev domain truncation is inferior to Fourier domain truncation for solving problems on an infinite interval'' \cite{Boyd88}.  We expect that with, albeit nontrivial, modifications, the results derived here will apply when using Fourier interpolants instead. Specifically, we conjecture, for example, that the analogous resolution result in the case of $\psiE$ would follow simply from the figure given in~\eqref{eqn:resolutionFigure_psiEwithoutc} multiplied by $2/\pi$.  This is an important topic for future work.  Alternatively, one may avoid domain truncation altogether, and consider either $\mathrm{sinc}$ interpolants, or rational Chebyshev approximation on the whole line.  These will also be considered in a future work.
 
Another topic for future investigations is that of the mapping parameters $\alpha$ and $L$.  In this paper, we have investigated three such choices: (i) $\alpha$ fixed and $L$ varying, (ii) $\alpha$ varying and $L$ fixed, and (iii) $\alpha$ and $L$ both varying.  As shown, the first leads to good convergence but poor resolution, the second leads to near-optimal resolution and, in the case of $\phiS$, only slightly worse convergence, and with the third, one obtains optimal resolution but forgoes classical convergence.  We have made no attempt to determine, either analytically or numerically, optimal choices for $\alpha$ and $L$.  Whilst we suspect that optimal choices may bring only marginal, and possibly function-dependent benefits, this is nonetheless an interesting question to address.

Finally, we remark that we have only examined \emph{exponential} transforms in this paper. However, there exists a further class of transformations called \emph{double-exponential} (DE) transforms which also have application within our variable-transform framework. DE transforms elicit double-exponential, rather than exponential, decay of their respective transplanted functions, and were first studied by Takahasi and Mori in~\cite{TakahasiMori74} and also later in e.g.~\cite{TanakaSugiharaMurotoMori09,TanakaSugiharaMuroto09}. Their advantage over exponential transforms is that they typically obtain faster rates of convergence, though this comes at the cost of more stringent analyticity requirements. To the best of our knowledge, a resolution analysis of such maps has not been performed, so we leave this as future work as well.

\vspace{1em}
\textbf{Acknowledgements.} We would like to thank Nick Trefethen and Daan Huybrechs, both of whom read an early draft of this manuscript and provided helpful comments. We additionally credit Nick Trefethen with proposing the use of a conformal map onto an infinite slit-strip.

\end{document}